\documentclass[reqno,11pt]{amsart}
\usepackage{amsthm,amsmath,amsfonts,amssymb}
\usepackage{color}
\usepackage{bold-extra} 

 \usepackage[pagebackref]{hyperref}
\hypersetup{
  colorlinks=true,
  citecolor = blue, 
  linkcolor= blue,
}

\usepackage{mathrsfs} 

\usepackage{etoolbox}
\patchcmd{\section}{\normalfont}{\bfseries}{}{}
\makeatletter
\renewcommand{\@secnumfont}{\bfseries}
\makeatother

\usepackage[all]{xy}
\usepackage{pb-diagram,pb-xy} 

\usepackage{enumerate}
\usepackage{enumitem} 
\setlist[enumerate,1]  {label={\rm (\roman*)}, leftmargin=1.5em}   
\setlist{noitemsep} 

\newcommand{\R}{{\mathbb R}}
\newcommand{\N}{{\mathbb N}}
\newcommand{\C}{{\mathbb C}}
\newcommand{\Z}{{\mathbb Z}}

\newcommand{\eps}{\varepsilon}
\newcommand{\D}{\displaystyle}
\newcommand{\<}{\langle}
\renewcommand{\>}{\rangle}
\newcommand{\Car}{\mbox{$\mathcal{X}$}} 

\newcommand{\mydef}{\ensuremath{\stackrel{\text{def}}{:=}}}

\newtheorem{theorem}{Theorem}[section]

\newtheorem{remark}[theorem]{Remark}
\newtheorem{lemma}[theorem]{Lemma}
\newtheorem{definition}[theorem]{Definition}
\newtheorem{proposition}[theorem]{Proposition}
\newtheorem{corollary}[theorem]{Corollary}

\numberwithin{equation}{section}

\renewenvironment{proof}[1][Proof]{\noindent \textbf{#1.} }
{\  \rule{0.5em}{0.5em}\par \medskip}

\title{
Exponential decay for fractional Schrödinger parabolic problems
}

\thanks{Partially supported by Projects PID2019-103860GB-I00 and PID2022-137074NB-I00, MINECO, Spain.
\newline{$\phantom{ai}$ Key words and phrases: fractional Schr\"odinger operator, linear evolution problem, stability}.
\newline{$\phantom{ai}$ Mathematical Subject Classification
2020:\ 35J10, 35R11, 47D06, 35B35}.
}

\thanks{${}^{*}$Partially supported by Severo
Ochoa Grant CEX2019-000904-S funded by MCIN/AEI/10.13039/ 501100011033.}

\begin{document}

\date{\bf \today}
\maketitle

\begin{center}
{\sc Jan W. Cholewa}${}^1\ $  \ {\sc and} \ {\sc Anibal Rodriguez-Bernal}${}^\text{2}$
\end{center}

\makeatletter
\begin{center}
${}^{1}$Institute of
Mathematics\\
University of Silesia in Katowice\\ 40-007 Katowice, Poland\\ {E-mail:
jan.cholewa@us.edu.pl}
\\ \mbox{}
\\
${}^{2}$Departamento de An\'alisis Matem\'atico y  Matem\'atica Aplicada\\ Universidad
  Complutense de Madrid\\ 28040 Madrid, Spain\\ and \\
  Instituto de Ciencias Matem\'aticas \\
CSIC-UAM-UC3M-UCM${}^{*}$, Spain  \\ {E-mail:
arober@ucm.es}
\end{center}
\makeatother

\setcounter{tocdepth}{3}
\makeatletter
\def\l@subsection{\@tocline{2}{0pt}{2.5pc}{5pc}{}}
\def\l@subsubsection{\@tocline{2}{0pt}{5pc}{7.5pc}{}}
\makeatother

\begin{abstract}
We discuss exponential decay in $L^p(\R^N)$, $1\leq p \leq \infty$, of
solutions of a  fractional Schr\"odinger parabolic equation  with a
locally uniformly integrable potential. The exponential type of the
semigroup of solutions is  considered and its dependence in 
 $1\leq p \leq \infty$ is addressed. We characterise a large class
of potentials for which solutions decay exponentially.  

\end{abstract}

\section{Introduction}

In this paper we discuss exponential decay of solutions of fractional Schrödinger
  semigroups
\begin{equation}
  \label{eq:intro_schrodinger_parabolic}
  \begin{cases}
    u_{t} + (-\Delta)^{\mu} u + V(x) u =0, \quad  x\in \R^{N}, \ t>0
    \\
u(x,0)= u_{0}(x), \quad  x\in \R^{N}
  \end{cases}
\end{equation}
with $0< \mu \leq 1$, $u_{0} \in
L^{p}(\R^{N})$ with $1\leq p\leq \infty$.

The nonnegative potential $V\geq 0$ belongs to the locally
uniform space $L_U^{p_{0}}(\R^N)$, which,  for  $1\leq p_{0}<\infty$,  is composed
of the functions $V \in L_{loc}^{p_{0}} (\R^N)$ such that there exists
$C>0$ such that for all $x_0\in\R^N$
\begin{displaymath}
\int_{B(x_0,1)} |V|^{p_{0}}\leq C
\end{displaymath}
endowed with the norm
\begin{displaymath}
\|V \|_{L_U^{p_{0}} (\R^N)}=\sup_{x_0\in\R^N}\|V \|_{L^{p_{0}}
  (B(x_0,1))} .
\end{displaymath}
For ${p_{0}}=\infty$ we define  $L_U^\infty(\R^{N})=L^\infty(\R^{N})$.

For $\mu=1$ and  potentials $0\leq V \in L^{1}(\R^{N}) +
L^{\infty}(\R^{N})$ it was proved in \cite{Arendt-Batty} that the
exponential decay in $L^{p}(\R^{N})$ of solutions of
(\ref{eq:intro_schrodinger_parabolic}) holds iff $V$ is sufficiently
positive at infinity in the sense that
\begin{equation}\label{eq:intro_Arendt-Batty_condition}
\int_{G} V (x) dx=\infty
\end{equation}
for any open set  $G \subset \R^{N}$ containing arbitrarily large balls, that is,  such that for
any $r > 0$ there exists $x_0 \in \R^N$ such that the ball of radius
$r$ around $x_0$  is included in  $G$.

To extend this result to solutions of
(\ref{eq:intro_schrodinger_parabolic}) with $0<\mu \leq 1$ and $0\leq V
\in L_U^{p_{0}}(\R^N)$  we first prove in Section
\ref{sec:fractional-schrodinger-semigroup}  that if $p_{0}>\max\{\frac
N{2\mu}, 1\}$ then for any $1\leq p\leq \infty$,
(\ref{eq:intro_schrodinger_parabolic}) is well posed in
$L^{p}(\R^{N})$ and that it defines a contraction semigroup of solutions
$u(t,u_{0}) = S_{\mu, V}(t) u_{0}$. This semigroup is order
preserving, strongly continuous if $1\leq p <\infty$ and analytic if
$1<p<\infty$, see Proposition
\ref{prop:nonegative_potential}. Moreover, for smooth initial data,
solutions can be represented in terms of the fractional unperturbed
semigroup  (that is, when $V=0$) and the variation of constants
formula. Also, for nonnegative initial data, the larger the potential, 
the smaller the solution is, see Proposition
\ref{prop:VCF_with_uniform_potential}. If $V\in L^{\infty}(\R^{N})$
then the variation of constants formula applies to all initial data in
$L^{p}(\R^{N})$, see Proposition
\ref{prop:VCF_with_bounded_potential}, 
and therefore it is natural to consider potentials $0\leq V \in
L_U^{p_{0}}(\R^N)$ that can be approximated in $L_U^{p_{0}}(\R^N)$ by bounded ones and to
analyze the convergence of the corresponding semigroups, see
Proposition \ref{prop:approximation_by_bounded_functions} and
Theorem \ref{thm:convergence_of_semigroups}.

With these tools, in Section \ref{sec:exponential-decay} we address
the exponential decay of solutions of
(\ref{eq:intro_schrodinger_parabolic}) in $L^{p}(\R^{N})$ with $1\leq
p\leq \infty$ for potentials $0\leq V \in L_U^{p_{0}}(\R^N)$
satisfying (\ref{eq:intro_Arendt-Batty_condition}). For this, we first
 analyze how  the exponential decay, actually, the exponential type of the
semigroup $\{S_{\mu, V}(t)\}_{t\geq 0}$,  depends on  $1\leq
p\leq \infty$. We prove that if $0<\mu \leq 1$ then the exponential type
is either zero or it is negative  for all $1\leq p \leq \infty$
simultaneously. If $\mu =1$ we recover the
result that the exponential type is independent of  $1\leq p
\leq \infty$, see Theorem \ref{thm:exponential_type_independent_Lp}. Then we prove that  that if solutions
decay exponentially then  (\ref{eq:intro_Arendt-Batty_condition})
holds. Conversely, if the potential can be approximated by 
bounded ones, then  (\ref{eq:intro_Arendt-Batty_condition}) actually
characterises the exponential decay, see Theorem
\ref{thm:decay_and_bottom_spectrumL2}.

Finally, in Section \ref{sec:basic-results} we have included several
known results for the unperturbed problem
(\ref{eq:intro_schrodinger_parabolic}) with $V=0$.

In this paper we denote by $c$ or $C$ generic constants that may
change from line to line, whose value is not important for the
results.

Also we will denote $A \sim B$ to denote quantities (norms or
functions, for example) such that there exist positive  constants $c_{1}, c_{2}$
such that $c_{1} A \leq B \leq c_{2}A$. Similarly we use $A \lesssim
B$ to indicate $A\leq c B$.

\section{Basic results on fractional operators and semigroups}
\label{sec:basic-results}

In this section we review several known results for the case of the
unperturbed equation, that is,  (\ref{eq:intro_schrodinger_parabolic})  with $V=0$.

First, for the standard heat equation, $\mu=1$, we have that for a wide
class of initial data, including, $L^{p}(\R^{N})$ for $1\leq p\leq
\infty$, the solution of (\ref{eq:intro_schrodinger_parabolic}) is
given by
\begin{equation} \label{eq:semigroup_kernel_L1U}
S(t)u_{0} (x) = \int_{\R^{N}} k(t,x ,y) u_{0}(y) \, dy  \quad t>0, \
x\in \R^N,
\end{equation}
for the selfsimilar convolution kernel
\begin{equation}\label{eq:kernel-heat-smgp}
0\leq k(t,x,y)= \frac{1}{t^{\frac{N}{2}}} k_{0}
    \left(\frac{x-y}{t^{\frac{1}{2}}} \right)= \frac{1}{(4\pi
      t)^\frac{N}{2}} e^{-\frac{|x-y|^2}{4t}} .
\end{equation}

In particular,  (\ref{eq:semigroup_kernel_L1U}) defines an  order preserving
semigroup  of contractions in $L^p(\R^N)$, $1\leq p\leq \infty$ and
 \begin{displaymath}
  \|S(t)  \|_{\mathcal{L}(L^p(\R^N))} = 1 , \quad t > 0 .
\end{displaymath}
This semigroup is strongly continuous if $1\leq p< \infty$ and
analytic if $1<p<\infty$. In particular, if  $1\leq p < \infty$ then for $u_0\in L^p(\R^N)$ we have
\begin{displaymath}
S(t)u_0\to u_0 \ \text{ as } \ t\to 0^+ \quad \mbox{in
    $L^p(\R^N)$ for any $1\leq p<\infty$}.
\end{displaymath}
  If $p=\infty$ then
 for $u_0\in L^\infty(\R^N)$
\begin{equation}\label{eq:behavior-at-0-of-S{1}(t)}
  S(t)u_0 \to u_0 \ \text{ as } \ t\to 0^+ \quad \mbox{in
    $L^p_{loc}(\R^N)$ for any $1\leq p<\infty$}.
\end{equation}
For all these classical results see e.g.
\cite{Cazenave-Haraux}.

Finally, because of the regularity of the kernel, we have that for
$1\leq p\leq \infty$ and $u_0\in L^{p}(\R^N)$, $u(t) = S(t) u_{0}$ in
(\ref{eq:semigroup_kernel_L1U}) is of class
$C^{\infty}((0,\infty)\times \R^{N})$ and satisfies the heat equation
pointwise in $(0,\infty)\times \R^{N}$.

The semigroup of solutions possesses also several \emph{smoothing}
estimates among which we recall that
for $1\leq p\leq q \leq
\infty$ there exists a constant $c_{p,q}>0$ such that
\begin{displaymath}
  \|S(t)\|_{{\mathcal L}(L^p(\R^N), L^q(\R^N))} =
\frac{c_{p,q}}{  t^{\frac{N}{2}(\frac1p-\frac1q)}}, \quad  t>0 .
\end{displaymath}

The fractional powers of $-\Delta$ can be defined in many equivalent
ways, see \cite{Kwasnicki}. From the heat semigroup
(\ref{eq:semigroup_kernel_L1U}) it can be defined
as
\begin{displaymath}
(-\Delta)^{\mu}
\phi(x)=\frac{1}{\Gamma(-\mu)}
\int_0^{\infty}\left(S(t) \phi(x)-\phi(x)\right) \frac{d t}{t^{1+\mu}}
\end{displaymath}
for $0<\mu <1$, which coincides with the nonlocal operator defined by
\begin{equation} \label{eq:fractional_laplacian_nonlocal} 
(-\Delta)^{\mu} \phi(x)  =C_{N, \mu} \text { P.V. }
    \int_{\mathbb{R}^N} \frac{\phi(x)-\phi(z)}{|x-z|^{N+2\mu}} d z
\end{equation}
with $ C_{N, \mu}  =\frac{2^{2\mu} \mu \Gamma(\frac{N}{2}+\mu))}{\pi^{N / 2}
      \Gamma(1-\mu)}$, see \cite{DiNPV}. When considered in
    $L^{p}(\R^{N})$ for $1\leq p\leq \infty$,  the domain of $(-\Delta)^{\mu} $
    is the set of functions $\phi \in L^{p}(\R^{N})$ such that
    $(-\Delta)^{\mu} \phi \in L^{p}(\R^{N})$ with the natural  norm
    \begin{displaymath}
      \|\phi\|_{L^{p}} + \|(-\Delta)^{\mu} \phi \|_{L^{p}} ,
    \end{displaymath}
see \cite[Section A.2]{C-RB-scaling1} when
$1\leq p<\infty$ and  \cite[Appendix A]{C-RB-linear_morrey} when
$p=\infty$. For $1<p<\infty$ this domain  coincides with the    Bessel
space $H^{2\mu}_{p}(\R^{N})$. When $p=2$ these spaces will be denoted by $H^{2\mu}(\R^{N})$.

  Using \cite[Lemma 6.2]{C-RB-scaling1} we have in particular,
  that for $1<p<\infty$,
\begin{displaymath}
   \|(-\Delta)^{\frac{1}{2}} \phi\|_{L^{p}(\R^{N} )} \sim
          \sum_{j=1}^N\|\partial_j \phi\|_{L^{p}(\R^N)}  \quad \text{in
        $ H^{1}_p(\R^N)$},
\end{displaymath}
and  for $\mu= \frac{1}{2}+ \mu'\in(0,1)$
\begin{displaymath}
    \|(-\Delta)^{\mu} \phi\|_{L^{p}(\R^{N} )} \sim   \sum_{j=1}^N
    \|(-\Delta)^{\mu'} \partial_j\phi\|_{L^{p}(\R^{N} )}  .
\end{displaymath}
 In particular, for $p=2$ and  for $\mu= \frac{1}{2}+ \mu'\in(0,1)$
 \begin{displaymath}
   \|(-\Delta)^{\mu} \phi\|_{L^{2}(\R^{N} )} \sim
    \sum_{j=1}^N [\partial_j \phi]_{W^{2\mu',2}(\R^{N} )}
\end{displaymath}
where in  the right hand side we have the Gagliardo seminorms
  \begin{displaymath}
    [\phi]^{p}_{W^{\mu,p}(\R^N)} = \iint_{\R^{N} \times \R^{N}}
    \frac{|\phi(x) -\phi(y)|^{p}}{|x-y|^{N+\mu p}} \, dy dx.
  \end{displaymath}

The fractional semigroup, that is, the semigroup defined by the
fractional Laplacian, can be constructed from the heat semigroup for
$0<\mu <1$ and $1\leq p<\infty$  as
\begin{equation}\label{eq:S0alpha(t)-in-MU}
S_{\mu}(t) \phi = \int_{0}^{\infty} f_{t,\mu}(s) S(s) \phi \,
ds=\int_{0}^{\infty} f_{1,\mu}(s) S(st^\frac{1}{\mu}) \phi \, ds,
\quad t> 0
\end{equation}
and $S_\mu(0)\phi =\phi$ for  $\phi\in L^p(\R^N)$, where $f_{t,\mu}$ is given by
\begin{displaymath}
0\leq   f_{t,\mu} (\lambda) =
  \begin{cases}
    \frac{1}{2\pi i} \D \int_{\sigma -i \infty}^{\sigma + i \infty}
    e^{z\lambda -tz^{\mu}} \, dz & \lambda \geq0,
    \\
    0 & \lambda <0,
  \end{cases}
\end{displaymath}
$\sigma >0$, and the branch for $z^{\mu}$ is chosen such that
$Re(z^{\mu})>0$ if $Re(z)>0$, and we have $\int_{0}^{\infty}
f_{t,\mu}(s) \, ds =1$,  see \cite[Appendix A.3]{C-RB-scaling1},
\cite[p. 259]{1978_Yosida} and \cite[(20'), p. 264]{1978_Yosida}. This
semigroup is  bounded, strongly continuous and analytic, see
\cite[Corollary 4.1]{C-RB-scaling1} and  \cite[Appendix
A.3]{C-RB-scaling1}. For $p=\infty$ the same construction is
possible, although the semigroup is bounded and  analytic but not strongly
continuous, see  \cite[Proposition B.1]{C-RB-linear_morrey}.  In particular, for
$1\leq p\leq \infty$ and $\mu\in(0,1)$, as the semigroup is analytic,
see \cite[Proposition 2.1.1]{lunardi95:_analy}, we have that
$u(t)=S_\mu(t)u_0$ with $u_0\in L^p(\R^N)$ solves
\begin{displaymath}
    u_{t} +(-\Delta)^\mu u =0 , \quad x \in \R^{N}, \ t>0, \qquad u(0)=u_0
  \end{displaymath}
  and   for  $1\leq p < \infty$ we have
\begin{displaymath}
S_\mu(t)u_0\to u_0 \ \text{ as } \ t\to 0^+ \quad \mbox{in
    $L^p(\R^N)$}.
\end{displaymath}

When $p=\infty$ the following holds.

\begin{proposition}

For $u_0\in L^\infty(\R^N)$ and $0<\mu<1$
\begin{displaymath}
  S_\mu(t)u_0 \to u_0 \ \text{ as } \ t\to 0^+ \quad \mbox{in
    $L^p_{loc}(\R^N)$ for any $1\leq p<\infty$}.
\end{displaymath}

\end{proposition}
\begin{proof}
  Using the second expression in (\ref{eq:S0alpha(t)-in-MU}) and
  \cite[formula (14), p. 262]{1978_Yosida}, if  $u_0\in
  L^\infty(\R^N)$ and we take a ball $B\subset \R^N$, using
  (\ref{eq:behavior-at-0-of-S{1}(t)}) and Lebesgue's dominated convergence theorem, we get
\begin{displaymath}
\|S_\mu(t)u_0 - u_0 \|_{L^p(B)} \leq \int_0^\infty f_{1,\mu} (s)
\|S(st^\frac{1}{\mu})u_0 - u_0 \|_{L^p(B)} \to 0 \ \text{ as } \
t\to0^+.
\end{displaymath}
\end{proof}

  Also, from (\ref{eq:S0alpha(t)-in-MU}) and the properties of the
  heat semigroup, it is easy to obtain that the
  fractional semigroup is of contractions and order preserving in
  $L^{p}(\R^{N})$ for $1\leq p\leq \infty$. Actually $ \|S_\mu(t)
  \|_{\mathcal{L}(L^p(\R^N))} = 1$ for all $t\geq 0$, see e.g.
  \cite[Corollary 3.7]{C-RB-scaling1}. Moreover it satisfies,
  among other,  the   smoothing estimates
\begin{equation}\label{eq:fractional_smoothing_LpLq}
  \|S_{\mu}(t)\|_{{\mathcal L}(L^p(\R^N), L^q(\R^N))} =
\frac{c_{p,q,\mu}}{  t^{\frac{N}{2\mu}(\frac1p-\frac1q)}}, \quad  t>0 ,
\end{equation}
for  $1\leq p\leq q \leq \infty$, see
\cite[Theorem 6.2]{2017BonforteSireVazquez:_optimal-existence}
and \cite[Proposition 6.5]{C-RB-scaling1}.

We also have for $0\leq \gamma \leq \gamma'\leq 1$, $1<p\leq q <\infty$
\begin{equation}\label{eq:fractional_smoothing_Sobolev}
  \|S_{\mu}(t)\|_{{\mathcal L}(H^{2\gamma}_p(\R^N),
    H^{2\gamma'}_q(\R^N))}
  \leq
  \frac{1}{t^{\frac{N}{2m\mu}(\frac1p-\frac1q)}} \max\{c_{p,q,\mu} ,
  \frac{c_{\gamma,\gamma',p,q,\mu}}{t^{\frac{\gamma'-\gamma}{\mu}}}  \}
  , \quad  t>0.
\end{equation}

The fractional semigroup has a $C^{\infty}$, positive convolution self
similar kernel such that
\begin{equation} \label{eq:fractional_semigroup_kernel_L1U}
S_{\mu}(t)u_{0} (x) = \int_{\R^{N}} k_{\mu}(t,x ,y) u_{0}(y) \, dy  \quad t>0, \
x\in \R^N,
\end{equation}

\begin{equation}\label{eq:kernel-fractional-heat-smgp}
  0\leq k_{\mu}(t,x,y)= \frac{1}{t^{\frac{N}{2\mu}}} k_{0,\mu}
    \left(\frac{x-y}{t^{\frac{1}{2\mu}}} \right)  \sim
\frac{1}{t^\frac{N}{2\mu}}
H_{\mu}\left(\frac{x-y}{t^{\frac{1}{2\mu}}}\right),
\end{equation}
where
\begin{equation}\label{eq:Hmu(z)}
H_{\mu}(z)=  \min\left\{1, \frac{1}{|z|^{N+2\mu}}\right\}
\sim I_{\mu}(z) = \frac{1}{(1+|z|^{2})^{\frac{N+2\mu}{2}}}, \quad z
\in \R^{N} .
\end{equation}
    The profile $k_{0,\mu}$ is even and satisfies
  \begin{displaymath}
    (-\Delta)^{\mu} k_{0,\mu}= \frac{x}{2\mu}\nabla k_{0,\mu} +
    \frac{N}{2\mu} k_{0,\mu} ,
  \end{displaymath}
see \cite[Section 6]{C-RB-scaling3-4},
\cite{2019Bogdan_fract_laplac_hardy,2017BonforteSireVazquez:_optimal-existence}.
Since the kernel is $C^{\infty}$ so the solution of the fractional
heat equation in (\ref{eq:fractional_semigroup_kernel_L1U}) is
$C^{\infty}((0,\infty)\times \R^{N})$, see
\cite{2017BonforteSireVazquez:_optimal-existence}.

\section{The fractional Schrödinger semigroup}
\label{sec:fractional-schrodinger-semigroup}

In this section we study the perturbed equations
\begin{displaymath}
  \begin{cases}
    u_{t} + (-\Delta)^{\mu} u + V(x) u =0, \quad  x\in \R^{N}, \ t>0
    \\
u(x,0)= u_{0}(x), \quad  x\in \R^{N}
  \end{cases}
\end{displaymath}
for potentials in the uniform space $L_U^{p_{0}}(\R^N)$, wit
$p_{0}\geq 1$, that is,
satisfying
\begin{displaymath}
  \|V \|_{L_U^{p_{0}}(\R^N)}=\sup_{x_0\in\R^N}\|V
  \|_{L^{p_{0}}(B(x_0,1))} < \infty.
\end{displaymath}

\begin{proposition}\label{prop:nonegative_potential}

 Assume $0\leq V \in L^{p_{0}}_{U}(\R^{N})$, for $p_{0} >\frac{N}{2\mu}$.

 Then $-(-\Delta)^\mu - V(x)$  defines an order preserving
 semigroup of contractions in
 $L^{p}(\R^{N})$ for any $1\leq p\leq  \infty$ which is strongly continuous
 if $1\leq p<\infty$ and  analytic  for
 $1<p<\infty$ and that we denote $\{S_{\mu,V }(t)\}_{t\geq 0}$.

 These semigroups are consistent in the sense that the semigroup in
 $L^{p}(\R^{N})$ and in $L^{q}(\R^{N})$  give the same result for any
 $t>0$ if $u_{0}\in L^{p}(\R^{N}) \cap L^{q}(\R^{N})$.

 Also the semigroup in $L^{p'}(\R^{N})$ is the adjoint of the
 semigroup in $L^{p}(\R^{N})$ for $1\leq p< \infty$ and  $\frac{1}{p}+
 \frac{1}{p'}=1$. 

\end{proposition}
\begin{proof}
 Consider the bilinear symmetric, nonnegative
 definite form
 \begin{displaymath}
   a(\phi, \psi) = \int_{\R^{N}} (-\Delta)^{\frac{\mu}{2}} \phi
   (-\Delta)^{\frac{\mu}{2}} \psi + \int_{\R^{N}} V(x) \phi \psi .
 \end{displaymath}

We first prove that  is well  defined in the Bessel space $H^{\mu}(\R^{N})$. For
this,   denote by $\{Q_{i}\}$ the family of cubes centered at points of
integer coordinates in $\R^{N}$ and with edges of length 1 parallel to
the axes. Thus if $V \in L^{p_{0}}_{U}(\R^{N})$,  for any
$\phi,\psi  \in  H^{\mu}(\R^{N})$
we have, using H\"older's inequality,
$$
\int_{\R^{N}} |V(x) \phi \psi|  =
\sum_{i} \int_{Q_{i}}  |V(x)| |\phi| |\psi| \leq \sum_{i}
\|V\|_{L^{p_{0}}(Q_{i})} \|\phi\|_{L^{r}(Q_{i})} \|\psi\|_{L^{r}(Q_{i})}
$$
with $\frac{1}{p_0} + \frac{2}{r}=1$ and $H^{\mu}(\R^{N}) \subset
L^{r}(\R^{N})$, that is   $\mu-\frac{N}{2} \geq -\frac{N}{r}$. This implies in
turn that we must have   $p_{0} \geq  \frac{N}{2\mu}$ (with strict
inequality if $p_{0}=1$ which implies $r=\infty$). 
Hence, using the  embedding
$H^{\mu}(Q_{i}) \subset L^{r}(Q_{i})$, with constants
independent of $i$, we get
\begin{displaymath}
  \int_{\R^{N}} |V(x)| |\phi | |\psi| \leq C \|V\|_{L^{p_{0}}_{U}(\R^{N})} \sum_{i}
\|\phi\|_{L^{r}(Q_{i})} \|\psi\|_{L^{r}(Q_{i})}
\end{displaymath}
\begin{displaymath}
\leq
C \|V\|_{L^{p_{0}}_{U}(\R^{N})}   \Big(\sum_{i}
\|\phi\|^{2}_{H^{\mu}(Q_{i})}\Big)^{\frac12} \Big(\sum_{i}  \|\psi\|^{2}_{H^{\mu}(Q_{i})}\Big)^{\frac12} \leq C
  \|V\|_{L^{p_{0}}_{U}(\R^{N})} \|\phi\|_{H^{\mu}(\R^{N})} \|\psi\|_{H^{\mu}(\R^{N})} ,
\end{displaymath}
by Lemma \ref{lem:get_nomrs_in_cubes} below.

Now  if $\phi \in H^{\mu}(\R^{N})$ then $|\phi| \in
H^{\mu}(\R^{N})$. To see this  we will use the generalised
 Strook-Varopoulos inequality, see \cite[Lemma
 5.2]{2012_PabloQuirosRodVazquez:_gener_fract_porous_medium_equat},
so  for $g(\phi), (-\Delta)^{\mu} \phi \in L^{2}(\R^{N})$
 \begin{equation} \label{eq:general_strook_varopoulos} \int_{\R^{N}}
   g(\phi) (-\Delta)^{\mu} \phi \geq \int_{\R^{N}} |
   (-\Delta)^{\frac{\mu}{2}} G(\phi)|^{2} \geq 0
 \end{equation}
 with $g'= (G')^{2}$. In particular, with $g(s)= s$ and
 $G(s) = |s|$ we get for
 $\phi, (-\Delta)^{\mu} \phi \in L^{2}(\R^{N})$, that is, for $\phi
 \in H^{2\mu}(\R^{N})$
 \begin{displaymath}
   \int_{\R^{N}} \phi (-\Delta)^{\mu} \phi \geq
   \int_{\R^{N}} | (-\Delta)^{\frac{\mu}{2}}
   |\phi||^{2} \geq 0
 \end{displaymath}
and we get, for $\phi \in H^{2\mu}(\R^{N})$
 \begin{displaymath}
   \int_{\R^{N}} |  (-\Delta)^{\frac{\mu}{2}} |\phi||^{2} \leq
   \int_{\R^{N}} \phi (-\Delta)^{\mu} \phi  =   \int_{\R^{N}}
   |(-\Delta)^{\frac{\mu}{2}} \phi|^{2} .
 \end{displaymath}
 For $\phi \in H^{\mu}(\R^{N})$, by density, if $\phi_{n} \in H^{2\mu}(\R^{N})$
  converges in $H^{\mu}(\R^{N})$ to $\phi$, the inequality above
  and $\phi_{n} \to \phi$ in $L^{2}(\R^{N})$ we can assume  that
  $|\phi_{n}| \to |\phi|$ weakly in $H^{\mu}(\R^{N})$ and we get for  $\phi \in H^{\mu}(\R^{N})$

 \begin{displaymath}
   \int_{\R^{N}} |  (-\Delta)^{\frac{\mu}{2}} |\phi||^{2} \leq
   \int_{\R^{N}}
   |(-\Delta)^{\frac{\mu}{2}} \phi|^{2} .
 \end{displaymath}
Thus $|\phi| \in H^{\mu}(\R^{N})$.

 In particular, if $V\geq 0$ we have
 \begin{equation}\label{eq:bilinear_1}
   a(|\phi|, |\phi|) \leq a(\phi, \phi), \qquad \phi \in
   H^{\mu}(\R^{N}).
 \end{equation}

 On the other hand, if $0\leq \phi \in H^{\mu}(\R^{N})$ then
 $v= \min\{\phi, 1\} \in H^{\mu}(\R^{N})$. To see this observe that
 $v=g(\phi)$ with $g$ Lipschitz and $g'(s)= \Car_{[0,1]}(s)$. Then in
 (\ref{eq:general_strook_varopoulos}) we have $G(s)=g(s)$ and for
 smooth $\phi$
 \begin{displaymath}
   \int_{\R^{N}} |  (-\Delta)^{\frac{\mu}{2}} g(\phi)|^{2} \leq
   \int_{\R^{N}}  g(\phi)  (-\Delta)^{\mu} \phi = \int_{\R^{N}}
   (-\Delta)^{\frac{\mu}{2}} g(\phi)   (-\Delta)^{\frac{\mu}{2}}
   \phi
 \end{displaymath}
 and Hölder's inequality gives
 \begin{displaymath}
   \int_{\R^{N}} |  (-\Delta)^{\frac{\mu}{2}} g(\phi)|^{2} \leq \int_{\R^{N}}
   |(-\Delta)^{\frac{\mu}{2}} \phi|^{2}, \quad \phi \in
   H^{2\mu}(\R^{\N}) .
 \end{displaymath}
Thus $g(\phi) \in H^{\mu}(\R^{N})$.  By density, if $\phi_{n} \in H^{2\mu}(\R^{N})$
  converges in $H^{\mu}(\R^{N})$ to $\phi$, the inequality above
  and $g(\phi_{n}) \to g(\phi)$ in $L^{2}(\R^{N})$ we can assume  that
  $g(\phi_{n}) \to g(\phi)$ weakly in $H^{\mu}(\R^{N})$ and we
  extend the inequality to $\phi \in H^{\mu}(\R^{N})$.

In particular,  if $V\geq 0$ we have
 \begin{equation}\label{eq:bilinear2}
   a(v, v) \leq a(\phi, \phi), \qquad \phi \in
   H^{\mu}(\R^{N}).
 \end{equation}

   Then from (\ref{eq:bilinear_1}) and  (\ref{eq:bilinear2}),
   \cite[Theorem 1.3.2, pag 12]{davies89:_heat}  and
   \cite[Theorem 1.3.3, pag 14]{davies89:_heat}, we have an order
   preserving
 semigroup of contractions in $L^{p}(\R^{N})$ for
 $1\leq p \leq \infty$. Moreover, from \cite[Theorem 1.4.1, pag
 22]{davies89:_heat} the semigroup in $L^{p}(\R^{N})$ is the
 extension of the one in $L^{2}(\R^{N})$,  the semigroup in
 $L^{p}(\R^{N})$ is the adjoint of the semigroup in
 $L^{p'}(\R^{N})$ for $1<p\leq \infty$, where $\frac{1}{p}+
 \frac{1}{p'} =1$. Also, the semigroup is analytic for $1<p<\infty$
 and strongly continuous for $1\leq p<\infty$.
 \end{proof}

The following result, used above,   was proved in Lemma 2.4 in
\cite{ACDRB02_attractors}.

\begin{lemma}\label{lem:get_nomrs_in_cubes}

  Let $\{Q_{i}\}$ be the family of cubes centered at points of integer
  coordinates in $\R^{N}$ and with edges of length 1 parallel to the
  axes.

  Then for any $0\leq s\leq 2$ and $1<p<\infty$
$$
\sum_{i} \|\phi\|^{p}_{H^{s}_{p}(Q_{i})} \leq C
\|\phi\|^{p}_{H^{s}_{p}(\R^{N})} \qquad \mbox{for all} \quad
\phi \in H^{s}_{p} (\R^{N}).
$$
\end{lemma}

    The next result shows in particular that, for suitable  $p$, the semigroup
    $\{S_{\mu,V}(t)\}_{t\geq 0}$ above in $L^{p}(\R^{N})$ can be
    represented in terms of the  fractional semigroup
    $\{S_{\mu}(t)\}_{t\geq 0}$  and the variation of constants formula
    (a.k.a. Duhamel's principle), at least for smooth initial data.

    \begin{proposition}
      \label{prop:VCF_with_uniform_potential}

For $0\leq V \in L^{p_{0}}_U(\R^N)$ with $p_{0}>\max\{\frac N{2\mu}, 1\}$, the
semigroup $u(t)= S_{\mu,V}(t)u_{0}$ satisfies the following.

\begin{enumerate}
\item
  For $1<p <\infty$ and $p\leq p_{0}$, the operator $-(-\Delta)^{\mu} - V$ with
  domain $D((-\Delta)^{\mu} ) = H^{2\mu}_{p}(\R^{N})$ is a
sectorial operator and it is the generator of the analytic semigroup $S_{\mu,
  V}(t)$. Moreover, for  $u_{0}\in H^{2\mu}_{p}(\R^{N})$, we have
$u\in C([0,\infty), H^{2\mu}_{p}(\R^{N}))\cap C^{1}((0,\infty),
L^{p}(\R^{N}))$ and satisfies the variations of constants formula 
\begin{displaymath}
 u(t) = S_{\mu}(t) u_{0} - \int_{0}^{t}  S_{\mu}(t-s) V u(s) \,
 ds, \quad  t>0.
\end{displaymath}

\item
For $1\leq p \leq \infty$ and  $u_{0}\in
L^{p}(\R^{N})$,
  \begin{displaymath}
    |S_{\mu,V}(t) u_{0}| \leq S_{\mu,V}(t) |u_{0}|  \leq S_{\mu}(t)
    |u_{0}| .
  \end{displaymath}
In particular, the semigroup $S_{\mu,V}(t) $ satisfies the smoothing
estimates (\ref{eq:fractional_smoothing_LpLq}).

\item
For $1\leq p \leq \infty$ and   $0\leq u_{0}\in L^{p}(\R^{N})$,
 if $V_{1} \geq V_{2}$ then for $u_{0} \geq 0$ we have
 $0\leq S_{\mu,V_{1}}(t)u_{0} \leq S_{\mu,V_{2}}(t)u_{0}$.
\end{enumerate}
   \end{proposition}
   \begin{proof}
(i) We show that  the domain of $(-\Delta)^{\mu} + V$ coincides
with the domain $D((-\Delta)^{\mu} ) = H^{2\mu}_{p}(\R^{N})$ and is a
sectorial operator.
Actually,     denote by $\{Q_{i}\}$ the family of cubes centered at points of
  integer coordinates in $\R^{N}$ and with edges of length 1 parallel
  to the axes. Thus if $V \in L^{p_{0}}_{U}(\R^{N})$, for any
  $\phi \in H^{2\mu}_{p}(\R^{N})$ we have, using H\"older's inequality,
$$
\|V\phi \|_{L^{p}(\R^{N})}^{p} = \int_{\R^{N}} |V|^{p} |\phi |^{p} =
\sum_{i} \int_{Q_{i}} |V|^{p} |\phi |^{p} \leq \sum_{i}
\|V\|^{p}_{L^{p_{0}}(Q_{i})} \|\phi\|^{p}_{L^{r}(Q_{i})}
$$
provided we chose $r$ such that $\frac{1}{p} = \frac{1}{p_{0}} + \frac{1}{r}$, which is possible
since $p_{0}\geq p$.  Now we use the Sobolev  embedding
\begin{displaymath}
  \|\phi\|_{L^{r}(Q_{i})} \leq C
  \|\phi\|_{H^{2\mu s}_{p}(Q_{i})}
\end{displaymath}
with constants independent of $i$, for $-\frac{N}{r} \leq 2\mu s -\frac{N}{p}$,
which requires $\frac{N}{p_{0}} \leq 2\mu s $. Hence we get, using
Lemma \ref{lem:get_nomrs_in_cubes},
\begin{equation} \label{eq: potential_relatively_bounded_H2mus}
\|V\phi \|_{L^{p}(\R^{N})}^{p} \leq C \|V\|_{L^{p_{0}}_{U}(\R^{N})}^{p}
\sum_{i} \|\phi\|^{p}_{H^{2\mu s}_{p}(Q_{i})} \leq C
\|V\|_{L^{p_{0}}_{U}(\R^{N})}^{p} \|\phi\|^{p}_{H^{2\mu
    s}_{p}(\R^{N})} .
\end{equation}

Then, since $p_{0} >\frac{N}{2\mu}$ we can take $0<s<1$ above and then by
interpolation
\begin{displaymath}
  \|V\phi \|_{L^{p}(\R^{N})} \leq C \|V\|_{L^{p_{0}}_{U}(\R^{N})}
  \|\phi\|_{H^{2\mu s}_{p}(\R^{N})} \leq  C \|V\|_{L^{p_{0}}_{U}(\R^{N})}
  \|\phi\|^{s}_{H^{2\mu}_{p}(\R^{N})}  \|\phi
  \|_{L^{p}(\R^N)}^{1-s}
\end{displaymath}
and then, for any $\eps>0$,
\begin{equation} \label{eq:potential_relatively_bounded}
  \|V\phi \|_{L^{p}(\R^{N})} \leq
  \eps  \|(-\Delta)^{\mu} \phi \|_{L^{p}(\R^N)}
  + C_{\eps} \|\phi \|_{L^{p}(\R^N)} , \qquad \phi \in
  H^{2\mu}_{p}(\R^{N}).
\end{equation}
The sectoriality result follows from Theorem 2.1 in Chapter 3 in \cite{1983_Pazy}.
The regularity of  $u(t)= S_{\mu,V}(t)u_{0}$ when  $u_{0}\in
H^{2\mu}_{p}(\R^{N})$ also follows from Chapter 1 in
\cite{1983_Pazy}. In this case $u$ satisfies
\begin{displaymath}
  u_{t} = -(-\Delta)^{\mu} u - V u, \quad t>0
\end{displaymath}
and we get the integral representation of $u$ as $Vu\in C([0,\infty),
L^{p}(\R^{N}))$, see  Chapter 4, Section 4.2  in \cite{1983_Pazy}.

\noindent (ii)
Since $ S_{\mu,V}(t) $ is order preserving, it is enough to prove the
second inequality because  for  any order preserving  semigroup in $L^{p}(\R^{N})$,
since $-|u_{0}|\leq u_{0} \leq
  |u_{0}|$ we have
  \begin{displaymath}
    |S(t) u_{0}| \leq S(t) |u_{0}| .
  \end{displaymath}

  First, for $1<p\leq p_{0}$,   if $0\leq u_{0} \in H^{2\mu}_{p}(\R^{N})$ we know that $V\geq 0$, $u(t) =
S_{\mu,V}(t) u_{0}\geq 0$ and $S_{\mu}(t)$ is order preserving and
then by (i)  $\int_{0}^{t}  S_{\mu}(t-s) V u(s) \,        ds \geq 0$
and
\begin{displaymath}
       u(t) = S_{\mu}(t) u_{0} - \int_{0}^{t}  S_{\mu}(t-s) V u(s) \,
       ds\leq S_{\mu}(t) u_{0}, \quad  t>0
     \end{displaymath}
and (ii) holds for such initial data. By density we get the result for
$0\leq u_{0}\in L^{p}(\R^{N})$. For initial data in
$L^{p}(\R^{N})$ with $p_{0}<p<\infty$ the result follows again by
density from the case above.

     Finally, for  $p=\infty$ and $0\leq u_{0} \in L^{\infty}(\R^{N})$, consider
     $0 \leq \varphi \in L^{1}(\R^{N})$ and then
     \begin{displaymath}
       \< S_{\mu, V}(t) u_{0}, \varphi \>  =   \< u_{0}, S_{\mu, V}
       (t) \varphi       \>  \leq   \< u_{0}, S_{\mu}(t)\varphi \>   =
       \<S_{\mu} (t) u_{0}, \varphi \>
     \end{displaymath}
and we get the pointwise inequality.

Now it is clear that the estimates
(\ref{eq:fractional_smoothing_LpLq}) apply to $S_{\mu, V}(t)$.

\noindent (iii)
Again, first for $1<p\leq p_{0}$, if  $0\leq u_{0} \in H^{2\mu}_{p}(\R^{N})$ and $u_{i}(t)= S_{\mu,
  V_{i}}(t) u_{0} \geq 0$,  by (i)    $u_{2}(t)$  satisfies
\begin{displaymath}
  (u_{2})_{t} + (-\Delta)^{\mu} u_{2} + V_{1}(x) u_{2} +
  \big(V_{2}(x)- V_{1}(x)\big) u_{2}  =0, \qquad x\in \R^{N}, \quad
  t>0
\end{displaymath}
and then, since  $0 \leq  \big(V_{1}- V_{2}\big) u_{2} \in C([0,\infty),
L^{p}(\R^{N}))$,
\begin{displaymath}
  u_{2}(t) = S_{\mu, V_{1}}(t) u_{0} + \int_{0}^{t}  S_{\mu,
    V_{1}}(t-s) (V_{1}-V_{2}) u_{2}(s) \,   ds \geq S_{\mu, V_{1}}(t)
  u_{0}, \quad  t>0.
\end{displaymath}
By density,   (iii) is proved for  this range of $p$.  For initial data in
$L^{p}(\R^{N})$ with $p_{0}<p<\infty$ the result follows again by
density from the case above.

Finally, for  $p=\infty$ and $0\leq u_{0} \in L^{\infty}(\R^{N})$, consider
     $0 \leq \varphi \in L^{1}(\R^{N})$ and then
     \begin{displaymath}
       \< S_{\mu, V_{1}}(t) u_{0}, \varphi \>  =   \< u_{0}, S_{\mu, V_{1}}
       (t) \varphi       \>  \leq   \< u_{0}, S_{\mu, V_{2}}(t)\varphi \>   =
       \<S_{\mu, V_{2}} (t) u_{0}, \varphi \>
     \end{displaymath}
and we get the pointwise inequality.
\end{proof}

In the case of a bounded potential
 $0\leq V \in L^{\infty}(\R^{N})$,  part (i) in Proposition
 \ref{prop:VCF_with_uniform_potential} can be somehow improved
and for all $1\leq p\leq \infty$, the semigroup
    $\{S_{\mu,V}(t)\}_{t\geq 0}$ above in $L^{p}(\R^{N})$ can be
    represented by the variation of constants formula for all initial data. Observe
 that actually the sign condition on $V$ is not needed in the proof
 below.

 \begin{proposition}
   \label{prop:VCF_with_bounded_potential}

   Assume $0\leq V \in L^{\infty}(\R^N)$. Then for $1\leq p <\infty$
 the operator $-(-\Delta)^{\mu} - V$ with
  domain $D((-\Delta)^{\mu} ) = H^{2\mu}_{p}(\R^{N})$ is the generator
  of the semigroup $S_{\mu,
    V}(t)$.

  For $1\leq p\leq \infty$ and   $u_{0}\in L^{p}(\R^{N})$, we have
  that $u(t)= S_{\mu, V}(t)u_{0}$ satisfies the variations of constants formula 
     \begin{displaymath}
       u(t) = S_{\mu}(t) u_{0} - \int_{0}^{t}  S_{\mu}(t-s) V u(s) \,
       ds, \quad  t>0.
     \end{displaymath}

\end{proposition}
\begin{proof}
For $1\leq p <\infty$ the semigroup $\{ S_{\mu}(t)\}_{t\geq 0}$ is
strongly continuous and the multiplication by $V$ is a bounded
operator in $L^{p}(\R^{N})$. Hence the result, including the
variations of constants formula, follows from Section 3.1 in
\cite{1983_Pazy}.

Before dealing with the case $p=\infty$ observe first  that for $1\leq
p\leq \infty$ and $T_{0} \| V\|_{L^{\infty}}
<1$ the mapping
     \begin{displaymath}
       \mathcal{F}_{p}(u)(t) =  S_{\mu}(t) u_{0} - \int_{0}^{t}  S_{\mu}(t-s) V u(s) \,
       ds
     \end{displaymath}
     defines a contraction in  $L^{\infty}((0,T_{0} ), L^{p}(\R^{N}))$ as
     \begin{displaymath}
       \| \mathcal{F}_{p} (u)(t)\|_{L^{p}} \leq \|u_{0}\|_{L^{p}} +
       \int_{0}^{t}  \| V\|_{L^{\infty}} \|u(s)\|_{L^{p}} \,       ds
     \end{displaymath}
     so $\mathcal{F}_{p} (u)  \in L^{\infty}((0,T_{0} ), L^{p}(\R^{N}))$
     and
  \begin{displaymath}
       \| \mathcal{F}_{p} (u)(t) - \mathcal{F}_{p} (v)(t)\|_{L^{p}} \leq
       \int_{0}^{t}  \| V\|_{L^{\infty}} \|u(s)- v(s)\|_{L^{p}} \,
       ds \leq T_{0}   \| V\|_{L^{\infty}} \sup_{0\leq t \leq T_{0} }\|u(s) -
       v(s)\|_{L^{p}} .
     \end{displaymath}

Take $u_{1}(t)$ for $0\leq t \leq T_{0} $ the unique fixed point in
$L^{\infty}((0,T_{0} ), L^{p}(\R^{N}))$ and consider for $t\geq
T_{0}$,
\begin{displaymath}
       \mathcal{H}_{p}(u)(t) =  S_{\mu}(t-T_{0})  u_{1}(T_{0} )   - \int_{T_{0} }^{t}  S_{\mu}(t-s) V u(s) \,
       ds .
     \end{displaymath}
     With similar estimates as above it is immediate to get that this
     is also a contraction in $ L^{\infty}((T_{0} ,2T_{0} ),
     L^{p}(\R^{N}))$ and has a unique fixed point $u_{2}(t)$ for $T_{0} \leq t \leq 2T_{0} $. Then it is easy to get that the function
     \begin{displaymath}
       u(t)=
       \begin{cases}
         u_{1}(t), & 0\leq t \leq T_{0}  \\
          u_{2}(t), & T_{0} \leq t \leq 2T_{0}
       \end{cases}
     \end{displaymath}
is a fixed point of  $\mathcal{F}_{p}$ in $L^{\infty}((0,2T_{0} ),
L^{p}(\R^{N}))$. Proceeding by induction we get a fixed point of
$\mathcal{F}_{p}$  in  $L^{\infty}((0,kT_{0} ), L^{p}(\R^{N}))$ for
any $k\in \N$.
The uniqueness of the fixed point for $\mathcal{F}_{p}$  follows from
Gronwall's lemma since for two such fixed points and $t>0$
\begin{displaymath}
       \| u(t) - v(t)\|_{L^{p}} \leq       \int_{0}^{t}  \| V\|_{L^{\infty}} \|u(s)- v(s)\|_{L^{p}} \,
       ds .
     \end{displaymath}

     In particular, for $1\leq p<\infty$ and $T>0$,  $u(t) = S_{\mu,V}(t) u_{0}$ for
$0\leq t \leq T$ is the unique fixed point of $\mathcal{F}_{p}$   in $L^{\infty}((0, T),
L^{p}(\R^{N}))$ for all $T>0$.

Hence,  if $u_{0} \in L^{p}(\R^{N}) \cap L^{\infty}(\R^{N})$ with $p>
\frac{N}{2\mu}$ then the semigroup solution in  $L^{p}(\R^{N})$, $u(t)
= S_{\mu,V}(t) u_{0}$ satisfies, using the variations of constants formula and
(\ref{eq:fractional_smoothing_LpLq}),
\begin{displaymath}
       \| u(t)\|_{L^{\infty}} \leq \|u_{0}\|_{L^{\infty}} +  \int_{0}^{t}  \frac{1}{(t-s)^{\frac{N}{2p\mu}}}\| V\|_{L^{\infty}} \|u(s)\|_{L^{p}} \,       ds
     \end{displaymath}
and therefore for any $T>0$, $S_{\mu,V}(\cdot) u_{0} \in L^{\infty}((0,T),
L^{\infty}(\R^{N}))$ and so it is the fixed point of $\mathcal{F}_{\infty}$ in
$[0,T]$ in $L^{\infty}((0,T), L^{\infty}(\R^{N}))$.

Now for  $u_{0}\in L^{\infty}(\R^{N})$  take a
sequence $u_{0}^{n} \in L^{1}(\R^{N}) \cap L^{\infty}(\R^{N})\subset
L^{p}(\R^{N})$  with $p>\frac{N}{2\mu}$ such that $u_{0}^{n} \to u_{0}$ weak-* in
$L^{\infty}(\R^{N})$ (for example, take the truncation by zero of
$u_{0}$ outside the ball $B(0,n)\subset \R^{N}$). Then for any
$\varphi \in L^{1}(\R^{N})$  and $t\geq 0$, as $n\to \infty$ we have,
since $S_{\mu,V}(t)  \varphi  \in  L^{1}(\R^{N})$,
\begin{displaymath}
  \< S_{\mu,V}(t) u_{0}^{n}, \varphi \> =   \< u_{0}^{n}, S_{\mu,V}(t)
  \varphi \> \to \< u_{0}, S_{\mu,V}(t)  \varphi \> = \< S_{\mu,V}(t) u_{0},  \varphi \>
\end{displaymath}
i.e. $ S_{\mu,V}(t) u_{0}^{n} \to  S_{\mu,V}(t) u_{0}$ weak-* in
$L^{\infty}(\R^{N})$.

On the other hand, setting $u^{n}(t) = S_{\mu,V}(t) u_{0}^{n}$ and
$u(t)= S_{\mu,V}(t) u_{0} $,  we
have,
\begin{displaymath}
  \<u^{n}(t) , \varphi \> = \< u_{0}^{n}, S_{\mu}(t) \varphi\> -
  \int_{0}^{t} \< u^{n}(s),   V S_{\mu}(t-s)\varphi \> .
\end{displaymath}
Therefore, for fixed $t>0$ and for $0<s<t$ we have, as $n\to \infty$,
\begin{displaymath}
\< u^{n}(s),   V S_{\mu}(t-s)\varphi \>
\to \<  u(s),   V S_{\mu}(t-s)\varphi \>
\end{displaymath}
and for all
$n\in \N$
\begin{displaymath}
  |\< u^{n}(s),   V S_{\mu}(t-s)\varphi \> |\leq \|V
  S_{\mu}(t-s)\varphi \|_{L^{1}} = g(s)
\end{displaymath}
with $g\in L^{1}(0,t)$. Then Lebesgue's theorem implies that, as $n\to
\infty$, we get
\begin{displaymath}
  \< u(t), \varphi \> = \< u_{0}, S_{\mu}(t) \varphi\> -
  \int_{0}^{t} \<  u(s),   V S_{\mu}(t-s)\varphi \> .
\end{displaymath}
Hence, for all $T>0$,  $u(\cdot) =  S_{\mu,V}(\cdot) u_{0}\in
L^{\infty}((0,T), L^{\infty}(\R^{N}))$ satisfies, for all $\varphi \in
L^{1}(\R^{N})$,
\begin{displaymath}
   \< u(t), \varphi \> = \<S_{\mu}(t)  u_{0}, \varphi\> -
  \int_{0}^{t} \< S_{\mu}(t-s) Vu(s),    \varphi \> , \quad 0<t<T
\end{displaymath}
and therefore
\begin{displaymath}
  u(t) =  S_{\mu}(t) u_{0} - \int_{0}^{t}  S_{\mu}(t-s) V u(s) \,
       ds
\end{displaymath}
and
$u(\cdot) =  S_{\mu,V}(\cdot) u_{0}$ is the unique fixed
point  of $\mathcal{F}_{\infty}$ in
$[0,T]$ in $L^{\infty}((0,T), L^{\infty}(\R^{N}))$.
\end{proof}

In view of Propositions \ref{prop:VCF_with_uniform_potential} and
\ref{prop:VCF_with_bounded_potential}, given $1\leq p_{0}< \infty$ we
want to discuss the class of $V\in
L^{p_{0}}_{U}(\R^{N})$ that can be approximated in
$L^{p_{0}}_{U}(\R^{N})$ by bounded functions and to analyze the
convergence of the corresponding semigroups.
Then we have the following result. Notice that below we use the class
$\dot L^{p_{0}}_{U}(\R^{N})$ of functions in $L^{p_{0}}_{U}(\R^{N})$
such that the translations are continuous in the
$L^{p_{0}}_{U}(\R^{N})$ norm, that is, $V\in \dot
L^{p_{0}}_{U}(\R^{N})$ iff
\begin{displaymath}
  \|\tau_{y} V - V \|_{L^{p_{0}}_{U}(\R^{N})} \to 0, \quad \mbox{as
    $|y|\to 0$}
\end{displaymath}
where $y\in \R^{N}$ and  $\tau_{y} V (x)= V(x-y)$. This is a closed
proper subspace of $L^{p_{0}}_{U}(\R^{N})$, see \cite{A-C-D-RB}.

\begin{proposition}
  \label{prop:approximation_by_bounded_functions}

  \begin{enumerate}
  \item
    For $1\leq p_{0}< \infty$ a function
    $V\in L^{p_{0}}_{U}(\R^{N})$ can be approximated in
    $L^{p_{0}}_{U}(\R^{N})$ by bounded functions if and only if,
    defining
    \begin{displaymath}
      V_{M}(x) =
      \begin{cases}
        M & \mbox{if $V(x)>M$} \\
        V(x) &\mbox{if $-M\leq V(x)\leq M$},  \\
        -M & \mbox{if $V(x)<-M$}
      \end{cases}  \qquad M >0,
    \end{displaymath}
    we have $V_{M} \xrightarrow[M \to \infty]{} V$ in $L^{p_{0}}_{U}(\R^{N})$.

  \item
    If $V\in \dot L^{p_{0}}_{U}(\R^{N})$ then $V$ can be approximated
    in    $L^{p_{0}}_{U}(\R^{N})$  by bounded functions.

\item
  Assume $V$ is not too large at infinity in the sense that there
  exists $M_{0}$ such that
  \begin{displaymath}
    \lim_{|x| \to \infty} \int_{B(x,1)} |V(y)-V_{M_{0}}(y)|^{p_{0}} \,
    dy = 0 .
  \end{displaymath}
Then $V_{M}  \xrightarrow[M \to \infty]{}  V$ in $L^{p_{0}}_{U}(\R^{N})$.

\item
  The class of functions in $L^{p_{0}}_{U}(\R^{N})$ that can be
  approximated  in
$L^{p_{0}}_{U}(\R^{N})$ by bounded functions is a closed proper  subspace of
$L^{p_{0}}_{U}(\R^{N})$.
  \end{enumerate}
\end{proposition}
\begin{proof}
(i) Assume $V$ can be approximated in
    $L^{p_{0}}_{U}(\R^{N})$ by bounded functions. Then there
exists $W_{n} \in L^{\infty}(\R^{N})$ such that $W_{n} \xrightarrow[n \to \infty]{}  V$ in
$L^{p_{0}}_{U}(\R^{N})$. Letting  $M_{n}= \|W_{n}\|_{L^{\infty}}$, we
have, for every $x\in \R^{N}$,
\begin{displaymath}
   \int_{B(x,1)} |V(y)-V_{M_{n}}(y)|^{p_{0}} \,    dy =  \int_{B(x,1)
     \cap \{V\geq M_{n}\}} |V(y)-M_{n}|^{p_{0}} \,    dy +  \int_{B(x,1)
     \cap \{V\leq -M_{n}\}} |V(y)+M_{n}|^{p_{0}} \,    dy
\end{displaymath}
and thus
\begin{displaymath}
   \int_{B(x,1)} |V(y)-V_{M_{n}}(y)|^{p_{0}} \,    dy \leq
   \int_{B(x,1)} |V(y)-W_{n}(y)|^{p_{0}} \,    d y .
\end{displaymath}

Therefore  $V_{M_{n}}  \xrightarrow[n \to \infty]{}  V$ in $L^{p_{0}}_{U}(\R^{N})$. Since clearly
$\|V- V_{M}\|_{L^{p_{0}}_{U}(\R^{N})}$ decreases with $M$, then $V_{M}
 \xrightarrow[M \to \infty]{}  V$ in $L^{p_{0}}_{U}(\R^{N})$.

The converse is immediate.

\noindent (ii)
If $V\in \dot L^{p_{0}}_{U}(\R^{N})$ from the results in
\cite{A-C-D-RB} we know that the solution of the heat equation
satisfies  $S(t)V
 \xrightarrow[t \to 0^{+}]{}  V$ in $ L^{p_{0}}_{U}(\R^{N})$  and $S(t)V\in
L^{\infty}(\R^{N})$.

\noindent (iii)
Define the family of bounded continuous   functions in   $\R^{N}$,  $H_{M}(x) = \int_{B(x,1)}
|V(y)-V_{M}(y)|^{p_{0}} \,  dy$  which are decreasing in $M$ and
converge to $0$ as $M\to \infty$ uniformly in bounded sets, by
Lebesgue's theorem.

Then the assumption reads $\lim_{|x|\to
  \infty} H_{M_{0}}(x) =0$. Therefore, given $\eps>0$ there exists
$R>0$ such that for $|x|>R$ and $M> M_{0}$ we have
\begin{displaymath}
  0\leq H_{M}(x) \leq H_{M_{0}}(x) \leq \eps .
\end{displaymath}
This and the uniform  convergence $H_{M}(x)  \xrightarrow[M \to \infty ]{}   0$ for $|x|\leq R$ implies
$H_{M} \xrightarrow[M \to \infty ]{}  0$ uniformly in $\R^{N}$ and
hence  $V_{M} \xrightarrow[M \to \infty ]{}  V$ in
$L^{p_{0}}_{U}(\R^{N})$.

\noindent (iv)
The closedness  is immediate as this class is the closure of
$L^{\infty}(\R^{N})$ in $L^{p_{0}}_{U}(\R^{N})$.

To show the subspace is proper, consider a sequence $|x_{n}|\to
\infty$ with $|x_{n}-x_{m}|> 2$ for all $n,m \in \N$. Then define
$V(x)= \sum_{n} n^{\frac{N}{p_{0}}} \Car_{B(x_{n}, \frac{1}{n})}
(x)$. Clearly $\|V\|_{L^{p_{0}}(B(x,1))} \leq c$ for all $x\in \R^{N}$
since $B(x,1)$ contains at most one point of $\{x_{n}\}_{n}$, so
$V\in L^{p_{0}}_{U}(\R^{N})$. Also, for any $M>0$ if $n>M$ we have
\begin{displaymath}
  \|V-V_{M}\|_{L^{p_{0}}(B(x_{n},1))}^{p_{0}} \geq c\frac{(n^{\frac{N}{p_{0}}} -M)^{p_{0}}}{n^{N}} \xrightarrow[n \to \infty ]{} c>0
\end{displaymath}
and therefore $\|V-V_{M}\|_{L^{p_{0}}_{U} (\R^{N})}^{p_{0}} \geq c>0$ for
all $M>0$.
\end{proof}

\begin{remark}

  \begin{enumerate}
  \item
    Assume $V \in L^{p_{0}}_U(\R^N)$ with $p_{0}>1$ and can be
    approximated in $L^{p_{0}}_{U}(\R^{N})$ by bounded functions. Then
    $V$ can be approximated in $L^{1}_{U}(\R^{N})$ by bounded
    functions. To see this just observe that uniform spaces are
    nested, that is, $L^{p}_{U}(\R^{N}) \subset  L^{q}_{U}(\R^{N})$, if $p> q$.

\item
    Conversely, if $V \in L^{p_{0}}_U(\R^N)$ with $p_{0}>1$ and can be
    approximated in $L^{1}_{U}(\R^{N})$ by bounded functions, then for
    every $1< q <p_{0}$, $V$ can be approximated in
    $L^{q}_{U}(\R^{N})$ by bounded functions. To see this, just notice
    that like in bounded domains, we have the interpolation inequality
    \begin{displaymath}
      \|f\|_{L^{q}_{U}} \leq \|f\|_{L^{p_{0}}_{U}}^{1-\theta}
      \|f\|_{L^{1}_{U}}^{\theta}, \quad \frac{1}{q} =
      \frac{1-\theta}{p_{0}}+ \frac{\theta}{1} , \quad 0<\theta <1,
    \end{displaymath}
    for any $f\in L^{p_{0}}_U(\R^N)$.
  \end{enumerate}
\end{remark}

Now we analize the approximation of the corresponding semigroups. For
this assume  now $0\leq V \in L^{p_{0}}_U(\R^N)$ with $p_{0}>\max\{\frac
N{2\mu}, 1\}$ and consider the increasing sequence $0\leq
V_{M}\leq V$. Then, by Proposition \ref{prop:VCF_with_uniform_potential}, we
have, for $1\leq p \leq \infty$ and   $0\leq u_{0}\in L^{p}(\R^{N})$,
\begin{displaymath}
  0\leq S_{\mu,V}(t)u_{0} \leq S_{\mu,V_{M}}(t)u_{0}
\end{displaymath}
and for  $M_{1} \leq M_{2}$
\begin{displaymath}
0\leq S_{\mu, V_{M_{2}}}(t)u_{0} \leq S_{\mu,V_{M_{1}}}(t)u_{0}  .
\end{displaymath}

Hence, the monotonic decreasing  limit
\begin{displaymath}
  \tilde{S}_{\mu,V}(t)u_{0} \mydef \lim_{M\to \infty} S_{\mu,V_{M}}(t)u_{0}
\end{displaymath}
exists pointwise  and in $L^{p}(\R^{N})$ for all $t>0$  and satisfies $ 0\leq S_{\mu,V}(t)u_{0} \leq
\tilde{S}_{\mu,V}(t)u_{0}$,  $\tilde{S}_{\mu,V}(t+s)u_{0}=
\tilde{S}_{\mu,V}(t) \tilde{S}_{\mu,V}(s)u_{0}$ and $\lim_{t\to 0^{+}}
\tilde{S}_{\mu,V}(t)u_{0} =u_{0}$  in $L^{p}(\R^{N})$. Using the
positive and negative part of $u_{0}\in L^{p}(\R^{N})$ we can extend $
\tilde{S}_{\mu,V}(t)$ to an order preserving $C^{0}$ semigroup of
contractions in $L^{p}(\R^{N})$.

Now we prove that $\tilde{S}_{\mu,V}(t)= S_{\mu,V}(t)$ for
$t\geq 0$   if $V$ can be approximated in    $L^{p_{0}}_{U}(\R^{N})$
by bounded functions.

\begin{theorem}
  \label{thm:convergence_of_semigroups}

  Assume $0\leq V \in L^{p_{0}}_U(\R^N)$ with $p_{0}>\max\{\frac
N{2\mu}, 1\}$ can be approximated in    $L^{p_{0}}_{U}(\R^{N})$ by
bounded functions.

Then  $\tilde{S}_{\mu,V}(t)= S_{\mu,V}(t)$ for $t\geq 0$, as operators
in $L^{p}(\R^{N})$ for $1\leq p\leq \infty$.

\end{theorem}

\begin{proof}
  Take $0\leq u_{0} \in C^{\infty}_{c}(\R^{N})$ and
  $u_{M}(t) = S_{\mu, V_{M}}(t)u_{0}$ and
  $u(t) = S_{\mu, V }(t)u_{0}$ for $t\geq 0$. Since these semigroups are consistent
  in the Lebesgue spaces, it is enough to show that
  $u(t) = \lim_{M\to \infty} u_{M}(t)$ for $t>0$ in some of these spaces. Take
  $1< p\leq p_{0}$. As $u_{0}$ belongs to the domain of the generator,
  $H^{2\mu}_{p}(\R^{N})$,
  we have, by Propositions \ref{prop:VCF_with_uniform_potential} and
  \ref{prop:VCF_with_bounded_potential},
  \begin{displaymath}
    u_{M}(t) = S_{\mu}(t) u_{0} - \int_{0}^{t}  S_{\mu}(t-s) V_{M} u_{M}(s) \,
    ds, \quad    u(t) = S_{\mu}(t) u_{0} - \int_{0}^{t}  S_{\mu}(t-s) V u(s) \,
    ds, \quad  t>0.
  \end{displaymath}
Also  by general properties of semigroups in \cite{1983_Pazy}, we have
that $u_{M}(t), u(t) \in H^{2\mu}_{p}(\R^{N})$ for $t>0$.

  Now notice that (\ref{eq: potential_relatively_bounded_H2mus}), (\ref{eq:potential_relatively_bounded}) and the
  closed mapping theorem also imply that the graph norm of $L = (-\Delta)^{\mu} +V$ in
  $H^{2\mu}_{p}(\R^{N})$ is equivalent to the $H^{2\mu}_{p}(\R^{N})$
  norm, that is,
  \begin{displaymath}
    c_{0}\|\phi\|_{H^{2\mu}_{p}} \leq \|L\phi\|_{L^{p}} +
    \|\phi\|_{L^{p}}\leq c_{1}\|\phi\|_{H^{2\mu}_{p}} .
  \end{displaymath}

  This implies, using that $L u(t) = L  S_{\mu, V }(t)u_{0} =  S_{\mu,
    V }(t) Lu_{0}$ and the semigroup is of contractions, that
  \begin{displaymath}
    \|u(t)\|_{H^{2\mu}_{p}} \leq c \|u_{0}\|_{H^{2\mu}_{p}} , \quad
    t>0 .
  \end{displaymath}

  Now using the variations of constants formula  and adding and
  subtracting the term  $\int_{0}^{t}  S_{\mu}(t-s)
  V_{M}u(s)\, ds$, we get that $z_{M}(t)= u(t) -u_{M}(t)$
  satisfies
  \begin{displaymath}
    z_{M}(t)  =  -\int_{0}^{t}  S_{\mu}(t-s) \big(V-V_{M}\big) u(s) \,
    ds -  \int_{0}^{t}  S_{\mu}(t-s) V_{M} z_{M}(s) \,
    ds, \quad  t>0.
  \end{displaymath}

  Hence for $0<r<1$, using (\ref{eq:fractional_smoothing_Sobolev}),
  \begin{displaymath}
    \|z_{M}(t)\|_{H^{2\mu r}_{p}} \leq \int_{0}^{t}
    \frac{c}{(t-s)^{r}} \|(V-V_{M})u(s)\|_{L^{p}} \, ds + \int_{0}^{t}
    \frac{c}{(t-s)^{r}} \|V_{M} z_{M}(s)\|_{L^{p}} \, ds
  \end{displaymath}
and (\ref{eq: potential_relatively_bounded_H2mus}) yields
\begin{displaymath}
  \|z_{M}(t)\|_{H^{2\mu r}_{p}} \leq \int_{0}^{t}
  \frac{c}{(t-s)^{r}} \|V-V_{M}\|_{L^{p_{0}}_{U}}
  \|u(s)\|_{H^{2\mu}_{p}} \, ds + \int_{0}^{t}
  \frac{c}{(t-s)^{r}} \|V_{M}\|_{L^{p_{0}}_{U}}
  \|z_{M}(s)\|_{H^{2\mu r}_{p}} \, ds
\end{displaymath}
and then
\begin{displaymath}
  \|z_{M}(t)\|_{H^{2\mu r}_{p}} \leq C t^{1-r} \|V-V_{M}\|_{L^{p_{0}}_{U}}
  \|u_{0}\|_{H^{2\mu}_{p}} + C \int_{0}^{t}
  \frac{1}{(t-s)^{r}}
  \|z_{M}(s)\|_{H^{2\mu r}_{p}} \, ds .
\end{displaymath}

Then Henry's singular Gronwall Lemma 7.1.1 in \cite{1981_Henry}  implies that for any $T>0$
\begin{displaymath}
  \|z_{M}(t)\|_{H^{2\mu r}_{p}} \leq C(T)
  \|V-V_{M}\|_{L^{p_{0}}_{U}}  \xrightarrow[M \to \infty ]{} 0 , \quad
  0<t<T.
\end{displaymath}

 Therefore  $\tilde{S}_{\mu,V}(t)u_{0}= S_{\mu,V}(t)u_{0}$ for $t\geq
 0$ and $0\leq u_{0} \in C^{\infty}_{c}(\R^{N})$. By linearity and
 density we have the result.
\end{proof}

\section{Exponential decay}
\label{sec:exponential-decay}

In this section we want to characterise certain classes of non
negative potentials $0\leq V\in L^{p_{0}}_{U}(\R^{N})$ with  $p_{0}
>\frac{N}{2\mu}$ such that the
fractional Schrödinger semigroup $\{S_{\mu, V}(t)\}_{t\geq0}$ decays
exponentially in $L^{p}(\R^{N})$. Since the semigroup is of
contractions in $L^{p}(\R^{N})$ for $1\leq p\leq \infty$ then either
$\|S_{\mu, V}(t) \|_{\mathcal{L}(L^{p}(\R^{N}))} =1$ for
all $t>0$ or there exists $M \geq 1$, $a>0$ such that
\begin{equation} \label{eq:exponential_decay}
  \|S_{\mu, V} (t)\|_{\mathcal{L}(L^p(\R^N))} \leq  Me^{-at}, \quad
  t\geq0 .
\end{equation}
Actually if for $t_{0}>0$, $\alpha \mydef \|S_{\mu, V}(t_{0})
\|_{\mathcal{L}(L^{p}(\R^{N}))} <1$, then writing $t=kt_{0} + s$, with
$k\in \N$ and $s\in (0,t_{0}]$ and the semigroup property gives
 $\|S_{\mu, V}(t) \|_{\mathcal{L}(L^{p}(\R^{N}))} \leq
  \frac{1}{\alpha} e^{-a t}$ where
  $a=\frac{|\ln{(\alpha)^{-1}|}}{t_0}>0$,
  i.e. (\ref{eq:exponential_decay}). Notice that one obtains $M\geq 1$ even if the semigroup is
of contractions.

Then we define the exponential type of the semigroup
$\{S_{\mu, V} (t)\}_{t\geq 0} $ in $L^p(\R^N)$, $1\leq p\leq \infty$,  as
\begin{displaymath}
 \omega_{p} \mydef   \sup\{a \in \R, \ \|S_{\mu, V} (t)\|_{\mathcal{L}(L^p(\R^N))} \leq
 M e^{-at}, \quad t\geq 0,  \ \mbox{for some $M\geq 1$} \}  ,
\end{displaymath}
and $\omega_{p}\geq 0$, since the semigroup is of contractions. The exponential type is
related to the spectral bound  of the generator $A_{p} = -(-\Delta)^\mu - V$, since these semigroups are
order preserving, then for $1\leq p<\infty$,
\begin{displaymath}
 - \omega_{p} = s(A_{p}) \mydef  \sup\{Re(\lambda),  \
  \lambda \in \sigma(A_{p})\},
\end{displaymath}
see \cite{weis95:_stability_Lp,weis98:_stability_Lp}.

Then we have the following  result. Notice that below
we denote by $\|\cdot\|_{p\to q} = \|\cdot\|_{\mathcal{L}(L^{p},
  L^{q})}$.

\begin{theorem}
  \label{thm:exponential_type_independent_Lp}

  For $0\leq V \in L^{p_{0}}_U(\R^N)$ with $p_{0}>\max\{\frac N{2\mu}, 1\}$, the
semigroup $u(t)= S_{\mu,V}(t)u_{0}$ satisfies the following.

  \begin{enumerate}
  \item
    For $1\leq p<\infty$ and $t\geq 0$,
    \begin{displaymath}
      \| S_{\mu,V}(t)\|_{2 \to 2} \leq   \| S_{\mu,V}(t)\|_{p \to p} =\| S_{\mu,V}(t)\|_{p' \to p'}
      \leq   \| S_{\mu,V}(t)\|_{\infty \to \infty} = \| S_{\mu,V}(t)\|_{1 \to 1}
    \end{displaymath}

  \item
    For $0<\mu <1$ and  $1\leq p\leq \infty$, the exponential type satisfies
    \begin{displaymath}
\omega_{2} \geq \omega_{p} \geq \omega_{\infty} \geq \frac{\omega_{2}}{1+\frac{N}{4\mu}}
\geq 0. 
\end{displaymath}

\item 
For $\mu=1$, the  exponential type $\omega_{p}$ is independent of
      $1\leq p \leq \infty$.
    \end{enumerate}

      In particular, in  cases (ii) and (iii) above, either $\omega_{p}=0$ for all $1\leq p\leq \infty$ or
they are all positive  simultaneously. In such a case there exists a common decay rate
independent of $p$.

\end{theorem}
\begin{proof}
  (i) Since for $1\leq p< \infty$ the  semigroup in $L^{p'}(\R^{N})$ is the adjoint of the
 semigroup in $L^{p}(\R^{N})$ we get the equalities in the
 statement.

 Now, if $1\leq p< 2 < p'$ then by the Riesz-Thorin interpolation
 theorem we get $
 \| S_{\mu,V}(t)\|_{2 \to 2} \leq   \| S_{\mu,V}(t)\|_{p \to p} =\|
 S_{\mu,V}(t)\|_{p' \to p'}$.

 \noindent (ii)
From (i) we immediately get $\omega_{2}\geq \omega_{p} = \omega_{p'}\geq
\omega_{\infty} = \omega_{1}$. 
Since  also $\omega_{\infty} = \omega_{1} \geq 0$ then if $\omega_{2} =0$
then the result is proved. Therefore we just need to consider the case
 $\omega_{2} >0$, that is, the semigroup decays exponentially in
 $L^{2}(\R^{N})$. 

 Since the semigroup in $L^{\infty}(\R^{N})$ is order preserving, then
for  $t>0$, $\|S_{\mu,V}(t)\|_{\infty\to \infty} = \|S_{\mu,V} (t)
\mathbf{1}\|_{L^\infty(\R^N)}$ so we estimate below the sup norm of
the nonnegative  solution $S_{\mu,V} (t) \mathbf{1}$. For this fix
$x\in \R^{N}$ and denote by $\Car_{R}$ the 
characteristic function of the ball $B(x,R)$ and then 
\begin{equation} \label{eq:splitting_4_estimate}
  \big(S_{\mu, V} (t)\mathbf{1}\big)  (x) = \big(S_{\mu, V} (t)
  \Car_{R} \big)(x)  + \big( S_{\mu, V} (t) (1-\Car_{R})\big) (x)  .
\end{equation}
Now since  from part (ii) in Proposition
\ref{prop:VCF_with_uniform_potential} the semigroup $S_{\mu, V} (t) $
satisfies
(\ref{eq:fractional_smoothing_LpLq}) we have, for $t\geq 1$,
\begin{displaymath}
 \big( S_{\mu, V} (t) \Car_{R} \big)(x) = \big( S_{\mu, V} (1) S_{\mu,
   V} (t-1) \Car_{R} \big) (x)
\end{displaymath}
and therefore
\begin{displaymath}
0\leq   \big(S_{\mu, V} (t) \Car_{R}\big) (x) \leq \|S_{\mu, V} (1)\|_{2\to \infty}
  \|S_{\mu, V} (t-1)\|_{2 \to 2} \|\Car_{R}\|_{2}  = C  \|S_{\mu, V}
  (t-1)\|_{2 \to 2} R^{\frac{N}{2}} .
\end{displaymath}

Now, for the second term in (\ref{eq:splitting_4_estimate}), using
again  part (ii) in Proposition \ref{prop:VCF_with_uniform_potential},
the properties of the fractional heat kernel in
(\ref{eq:fractional_semigroup_kernel_L1U}),
(\ref{eq:kernel-fractional-heat-smgp}),  (\ref{eq:Hmu(z)}) and
changing variables,  we have 
\begin{displaymath}
  0 \leq \big( S_{\mu, V} (t) (1-\Car_{R})\big) (x)   \leq \big(
  S_{\mu} (t) (1-\Car_{R})\big) (x)  = \int_{\R^{N}} k_{0,\mu}(z)
  (1-\Car_{R}(x-t^{\frac{1}{2\mu}}z )) \, dz , 
\end{displaymath}
that is,
\begin{displaymath}
   0 \leq \big( S_{\mu, V} (t) (1-\Car_{R})\big) (x) \leq
   \int_{|z|>\frac{R}{t^{\frac{1}{2\mu}}}} k_{0,\mu}(z) \, dz  \sim
   \int_{|z|>\frac{R}{t^{\frac{1}{2\mu}}}}
   \frac{1}{(1+|z|^{2})^{\frac{N+2\mu}{2}}} \, dz  . 
\end{displaymath}
Using polar coordinates we get
\begin{displaymath}
     0 \leq \big( S_{\mu, V} (t) (1-\Car_{R})\big) (x) \lesssim
     \int_{\frac{R}{t^{\frac{1}{2\mu}} }}^{\infty} \frac{1}{r^{2\mu +1}} \,
     dr  \sim \frac{t}{R^{2\mu}}  
\end{displaymath}
and therefore, for $t\geq 1$ and $R>0$, 
\begin{displaymath}
  \|S_{\mu,V}(t)\|_{\infty\to \infty}  \leq  c\|S_{\mu, V}
  (t-1)\|_{2 \to 2} R^{\frac{N}{2}} + c \frac{t}{R^{2\mu}}  . 
\end{displaymath}
Minimizing this in $R>0$ we get, for $t\geq 1$, 
\begin{displaymath}
    \|S_{\mu,V}(t)\|_{\infty\to \infty}  \leq  c  \|S_{\mu, V}
  (t-1)\|_{2 \to 2}^{\frac{1}{1+\frac{N}{4\mu}}} t^{\frac{N}{N +4\mu}}
  . 
\end{displaymath}
Therefore 
\begin{displaymath}
\omega_{\infty} \geq \frac{\omega_{2}}{1+\frac{N}{4\mu}} . 
\end{displaymath}

\noindent (iii)
  For the case $\mu=1$ we proceed as in case (ii)  but since the
  kernel is the Gaussian we get 
  \begin{displaymath}
 0 \leq \big( S_{1, V} (t) (1-\Car_{R})\big) (x) \lesssim     \int_{|z|>\frac{R}{t^{\frac{1}{2}}}} e^{-|z|^{2}} \, dz =
     \int_{\frac{R}{t^{\frac{1}{2}} }}^{\infty}  r^{N-1} e^{-r^{2}} \,
     dr \lesssim       \int_{\frac{R}{t^{\frac{1}{2}} }}^{\infty}  r
     e^{-\eps r^{2}} \,     dr  \sim e^{-\eps \frac{R^{2}}{t}}
  \end{displaymath}
with $\eps <1$, so choosing $R= At$ with $A$ large,  we get a decay of order $e^{-\eps At}$ as fast as
we want. Also, from the estimate in case (ii) above, we have 
\begin{displaymath}
  0\leq \big(S_{1, V} (t) \Car_{R}\big) (x)  \leq c  \|S_{1, V}
  (t-1)\|_{2 \to 2} t^{\frac{N}{2}} 
\end{displaymath}
which is only a
polynomial correction of the decay of $\|S_{1,V}(t)\|_{2\to
  2}$. Therefore $\omega_{\infty} = \omega_{2}$ and the result is
proved. 
\end{proof}

\begin{remark}

  \begin{enumerate}
  \item
    The independence of the exponential type when $\mu=1$ holds even without
assuming       $V\geq 0$, see Theorem B.5.1 in \cite{simon82:_schroed} and references
therein.

\item
  When $0<\mu <1$ the same proof of part (ii) in Theorem
  \ref{thm:exponential_type_independent_Lp} but using $\|S_{\mu, V}
  (t-1)\|_{p \to p}$ with $1\leq p<\infty$ leads to
  \begin{displaymath}
    \omega_{p} \geq \omega_{\infty} \geq \frac{\omega_{p}}{1+\frac{N}{2p\mu}}
  \end{displaymath}
  and this combined with $\omega_{p}= \omega_{p'}$ give
  \begin{displaymath}
    \lim_{p\to 1} \omega_{p} = \omega_{1}= \omega_{\infty} = \lim _{p\to \infty} \omega_{p} .
  \end{displaymath}

  \item 
    Observe that for $1<p<\infty$, if  
    $\|S_{\mu, V}(t) \|_{\mathcal{L}(L^{p}(\R^{N}))} =1$ for all $t>0$
    then there are solutions that converge to 0 in $L^{p}(\R^{N})$ arbitrarily slow.

    More precisely, for any continuous function
    $g:[0,\infty) \to (0,1]$ such that $\lim_{t\to\infty}g(t)=0$,
    there exists $u_0\in L^{p}(\R^{N})$ such that $\|S_{\mu,
      V}(t)u_0\|_{L^{p}} \to 0$ as $t\to \infty$ and 
    \begin{displaymath}
      \limsup_{t\to\infty} \frac{\|S_{\mu, V}(t)u_0\|_{L^{p}}}{g(t)}=\infty .
    \end{displaymath}

    To see this, observe first that  $(-\Delta)^{\mu}$ is injective in
    $L^{p}(\R^{N})$ for $1\leq p<\infty$. Actually, from
    \cite{grafakos14:_class_fourier}, Chapter 2, Corollary 2.42, 
  page 125, we have that  for a tempered distribution such that
  $(-\Delta)^{\gamma} \phi =0$ then  
  the Fourier transform  $\hat
  \phi$ has support in $\{0\}$ and therefore $\phi$ is a
  polynomial.

  Then using Corollary 1.1.4 in \cite{2001_Martinez-Sanz} we get that
  if $1<p<\infty$ then $(-\Delta)^{\mu}$ has dense range in
  $L^{p}(\R^{N})$.  Since the semigroup $\{S_{\mu}(t)\}_{t\geq 0}$ is
  analytic and bounded  in  $L^{p}(\R^{N})$ for $1< p<\infty$, this  in turn implies, by Remark 1.5 in
  \cite{ArendtBattyBenilan92}, that for all $u_{0}\in L^{p}(\R^{N})$
  we have $\|S_{\mu}(t) u_{0}\|_{L^{p}}\to 0$ as $t\to \infty$. From
  part (ii) in Proposition \ref{prop:VCF_with_uniform_potential}  we
 also have  $\|S_{\mu, V}(t) u_{0}\|_{L^{p}}\to 0$ as $t\to \infty$.

    Now assume by contradiction that for all
    $u_0\in L^{p}(\R^{N})$ there exists $C=C(u_{0})$ such that
    \begin{displaymath}
      \frac{\|S_{\mu, V}(t)u_0\|_{L^{p}}}{g(t)} \leq C(u_{0}), \quad t\geq
      0.
    \end{displaymath}
    Then, the uniform bounded principle implies that, for some $M>0$,
    \begin{displaymath}
      \frac{\|S_{\mu, V}(t)\|_{\mathcal{L}(L^{p})}}{g(t)} \leq M  .
    \end{displaymath}
    But then for large enough $t$ we have
    $\|S_{\mu, V}(t)\|_{\mathcal{L}(L^{p})} \leq M g(t) <1$ which is a
    contradiction.
  \end{enumerate}
\end{remark}

We also have the following result that implies that if there is
exponential decay, then we can take $M=1$ in
(\ref{eq:exponential_decay}). 

\begin{proposition}\label{prop:alternative_contractions_or_decay}
Either
$\|S_{\mu, V}(t) \|_{\mathcal{L}(L^{2}(\R^{N}))} =1$ for
all $t>0$ or there exists $a>0$ such that for all $1\leq p\leq \infty$
\begin{displaymath}
  \|S_{\mu, V} (t)\|_{\mathcal{L}(L^p(\R^N))} \leq  e^{-at}, \quad
  t\geq0 .
\end{displaymath}
\end{proposition}

\begin{proof}
  Again, as above, if for some $t_{0}>0$ we have $\|S_{\mu, V}(t_{0})
\|_{\mathcal{L}(L^{2}(\R^{N}))} <1$,  then we have (\ref{eq:exponential_decay}) in
  $L^2(\R^N)$.  
  From this and the results in
  Section \ref{sec:fractional-schrodinger-semigroup} we see that the
  semigroup $\{e^{at} S_{\mu,V}(t)\}_{t\geq0}$ in $L^2(\R^N)$ is
  bounded, analytic, $C^0$, and is generated by
the selfadjoint operator  $\mathscr{A}= - (-\Delta)^\mu -V + a$ with 
the same domain   $D(\mathscr{A})$ as $- (-\Delta)^\mu -V$. In
particular,  
  $D(\mathscr{A}) =H^{2\mu}(\R^N)$ if $p_0\geq 2$.

  In particular, using \cite[Remark 5.4, Chapter 1]{1983_Pazy},  the
  right half space in $\C$ is in the resolvent of $\mathscr{A}$
  and therefore the  spectrum $\sigma(\mathscr{A})$ must be contained
  in $(-\infty,0]$. Then \cite[Proposition 12.8,
  p. 270]{2012_Schmudgen} implies 
\begin{displaymath}
\langle -\mathscr{A} \phi, \phi \rangle_{L^2(\R^N)} \geq 0 \quad \text{ for all } \phi \in D(\mathscr{A})
\end{displaymath}
as otherwise $-\mathscr{A}$
would have to possess a negative eigenvalue. 

This in turn implies, using \cite[Corollary 4.4]{1983_Pazy} that this
operator generates a semigroup of contractions, 
$\{e^{at} S_{\mu,V}(t)\}_{t\geq0}$. That is, we have $M=1$ in (\ref{eq:exponential_decay})
when $p=2$. See also \cite{KubruslyLevan2011}. 

But  since
\begin{displaymath}
\langle -\mathscr{A} \phi, \phi \rangle_{L^2(\R^N)}
=\int_{\R^{N}} |(-\Delta)^\frac{\mu}{2}\phi|^{2} + \int_{\R^{N}} (V-a)
|\phi|^2 \geq0, \quad \phi \in D(\mathscr{A})
\end{displaymath} 
this also implies that we can use Proposition
\ref{prop:nonegative_potential} with the potential $V-a$ to conclude
that $-\mathscr{A}=-(-\Delta)^\mu - V + a $  defines an order
preserving 
 semigroup of contractions in
 $L^{p}(\R^{N})$ for any $1\leq p\leq  \infty$, that is,  $\{e^{at}
 S_{\mu,V}(t)\}_{t\geq0}$ in $L^p(\R^N)$ is of contractions for any
 $1\leq p\leq  \infty$. 
\end{proof}

This result combined with Theorem
\ref{thm:exponential_type_independent_Lp} gives the following.

\begin{corollary}
\label{cor:alternative_contractions_or_decay}

Either
$\|S_{\mu, V}(t) \|_{\mathcal{L}(L^{p}(\R^{N}))} =1$ for
all $t>0$ and all $1\leq p\leq \infty$ or there exists $a>0$ such that for all $1\leq p\leq \infty$
\begin{displaymath}
  \|S_{\mu, V} (t)\|_{\mathcal{L}(L^p(\R^N))} \leq  e^{-at}, \quad
  t\geq0 .
\end{displaymath}
\end{corollary}
\begin{proof}
  If for some $1\leq p_{0} \leq \infty$ and $t_{0}>1$ we have
  $\|S_{\mu, V}(t_{0}) \|_{\mathcal{L}(L^{p}(\R^{N}))} <1$ then by
  Theorem \ref{thm:exponential_type_independent_Lp} we get
  $\omega_{2}>0$ and Proposition
  \ref{prop:alternative_contractions_or_decay} gives the result. 
\end{proof}

Now we give conditions on the potential to have exponential
decay. Following \cite{Arendt-Batty} we consider the following
\begin{definition}

The class $\mathcal{G}$ consists of all open subsets of $\R^N$
containing arbitrarily large balls, that is, the sets such that for
any $r > 0$ there exists $x_0 \in \R^N$ such that the ball of radius
$r$ around $x_0$  is included in this set.
\end{definition}

\begin{theorem}
  \label{thr:exponential_decay_4_fractional}

  Assume that $0\leq V \in L^{p_{0}}_U(\R^N)$ with  $p_{0}>\max\{\frac
  N{2\mu}, 1\}$ and there exists $M>0$ such that $0\leq
  V_{M}(x) = \min\{V(x), M\} $ satisfies
\begin{displaymath}
\int_{G} V_{M} (x) dx=\infty \quad \text{ for all } \ G\in \mathcal{G}.
\end{displaymath}

Then there exists  $a>0$, independent of $1\leq p \leq  \infty$,  such that
\begin{displaymath}
\|S_{\mu, V} (t)\|_{\mathcal{L}(L^p(\R^N))} \leq  e^{-at},
\quad t\geq0 .
\end{displaymath}

\end{theorem}
\begin{proof}
Since $0\leq    V_{M}(x) \leq V(x)$, by    Proposition \ref{prop:VCF_with_uniform_potential}, it
is enough to prove the decay for $S_{\mu,V_{M}}(t)$. Therefore we
can assume $V$ is bounded.

From Corollary \ref{cor:alternative_contractions_or_decay}, to  prove the
result  it is enough  to find $t>0$ such
that
\begin{displaymath}
\| S_{\mu,V}(t)\|_{\infty \to \infty} = \|S_{\mu, V} (t) \mathbf{1}\|_{L^\infty(\R^N)} <1 .
\end{displaymath}

For this,
notice that with  $u_{0} =\mathbf{1}$, we have in particular that,
from Proposition \ref{prop:VCF_with_bounded_potential},  $u(t)=
  S_{\mu,V} (t) \mathbf{1} $ satisfies
  \begin{displaymath}
u(t) = S_{\mu} (t) \mathbf{1} -   \int_{0}^{t} S_{\mu} (t-s) V u(s)
\, ds = \mathbf{1} -    \int_{0}^{t} S_{\mu} (t-s) V u(s) \, ds .
  \end{displaymath}
  Substituting the expresion above for $u(s)$ inside the integral term
  we get
  \begin{displaymath}
u(t) = \mathbf{1} -   \int_{0}^{t} S_{\mu} (t-s) V \, ds +
\int_{0}^{t} S_{\mu} (t-s) V  \int_{0}^{s} S_{\mu} (s-r) V u(r) \, dr
\, ds
  \end{displaymath}

  We use that $0\leq u(s) \leq S_{\mu} (s) \mathbf{1} = \mathbf{1} $ and
  $0\leq V\leq  \|V\|_{\infty} \mathbf{1}$ and then
  \begin{displaymath}
    \int_{0}^{t} S_{\mu} (t-s) V  \int_{0}^{s} S_{\mu} (s-r) V u(r) \,
    dr \, ds  \leq  \|V\|_{\infty} \int_{0}^{t} S_{\mu} (t-s) V
    \int_{0}^{s} S_{\mu} (s-r) \mathbf{1}  \, dr
\, ds
  \end{displaymath}
  \begin{displaymath}
    = \|V\|_{\infty} \int_{0}^{t} s S_{\mu} (t-s) V
\, ds  = \|V\|_{\infty} \int_{0}^{t} (t-s) S_{\mu} (s) V
\, ds \leq \|V\|_{\infty} t \int_{0}^{t} S_{\mu} (s) V
\, ds .
  \end{displaymath}

  Therefore
  \begin{displaymath}
    u(t) \leq  \mathbf{1} -   \int_{0}^{t} S_{\mu} (t-s) V \, ds + \|V\|_{\infty} t \int_{0}^{t} S_{\mu} (s) V
\, ds = \mathbf{1}  + (\|V\|_{\infty} t -1) \int_{0}^{t} S_{\mu} (t-s)
V \, ds .
  \end{displaymath}

From  the results in Section \ref{sec:basic-results}
\begin{displaymath}
(S_\mu(s) V)(x) =\int_{\R^N} k_{\mu}(s,x,y) V(y) \, dy
\end{displaymath}
where
$k_{\mu}(s,x,y)= \frac{1}{t^{\frac{N}{2\mu}}} k_{0,\mu}
    \left(\frac{x-y}{t^{\frac{1}{2\mu}}} \right)$ is as in
    (\ref{eq:kernel-heat-smgp}) and
    (\ref{eq:kernel-fractional-heat-smgp}). Using \cite[Proposition
    1.4]{Arendt-Batty} we see in turn
that there exist $c>0$ and $r>0$ such that
\begin{displaymath}
  \int_{B(x,r)} V (y) dy \geq c \ \text{ for all } \ x\in \R^N.
\end{displaymath}

For  any  $x\in \R^N$ and $s>0$ we have
\begin{displaymath}
\begin{split}
  S_\mu(s) V(x) &=\int_{\R^N} \frac{1}{s^\frac{N}{2\mu}}
k_{0,\mu}\left(\frac{x-y}{s^{\frac{1}{2\mu}}}\right)  V (y) \, dy
\\&
\geq \int_{B(x,r)} \frac{1}{s^\frac{N}{2\mu}}
k_{0,\mu}\left(\frac{x-y}{s^{\frac{1}{2\mu}}}\right) V(y) \, dy
\\&
\geq
  \inf_{|z|\leq r} \frac{1}{s^\frac{N}{2\mu}}
k_{0,\mu}\left(\frac{z}{s^{\frac{1}{2\mu}}}\right) \int_{B(x,r)} V(y) \, dy
  \\&
  \geq
c   \begin{Bmatrix}
   (4\pi s)^{-\frac{N}2} e^{-\frac{r^2}{4s}}, & \mu=1\\
  \frac{ s}{(s^\frac{1}{\mu} + r^2)^\frac{N+2\mu}{2}}
  , & 0<\mu<1
  \end{Bmatrix}
  =: g_\mu (s)
 \end{split}
  \end{displaymath}
 where we have used the estimate of fractional heat equation kernel from (\ref{eq:kernel-heat-smgp}) and (\ref{eq:kernel-fractional-heat-smgp})-(\ref{eq:Hmu(z)}).

 Denoting $\mathscr{G}_\mu(t)
  \mydef     \int_0^t g_\mu(s) \, ds$, we get,  for  $0< t <
  \frac{1}{\|V\|_{L^\infty(\R^N)}}$ and $x\in \R^{N}$
\begin{displaymath}
(S_{\mu, V} (t){\bf 1})(x) \leq {\bf 1} - (1-t
\|V\|_{L^\infty(\R^N)}) \mathscr{G}_\mu(t) <1 .
\end{displaymath}
\end{proof}

In case $V$ can be approximated by bounded potentials, we have the
following result.

\begin{proposition}
\label{prop:Arendt-Batty_4_bounded_approximations}

  Assume $V \in L^{1}_{U}(\R^{N})$ can be approximated in
  $L^{1}_{U}(\R^{N})$ by bounded functions.
  Then, they are equivalent

  \begin{enumerate}
\item
    $\int_{G} V_{M} =\infty$ for all $G\in \mathcal{G}$ with some
$M$.

  \item
    $\int_{G} V =\infty$ for all $G\in \mathcal{G}$.
\end{enumerate}
\end{proposition}
\begin{proof}
  Since $0\leq V_{M}\leq V$, clearly (i) implies (ii).

  Conversely, from \cite[Proposition 1.4]{Arendt-Batty},  (ii) is
  equivalent to the existence of   $c>0$ and $r>0$ such that
\begin{displaymath}
\inf_{x\in \R^N} \int_{B(x,r)} V \geq c.
\end{displaymath}
By a simple covering argument, independent of $x$, we get that  $\sup_{x\in\R^N}\| V_{M} - V \|_{L^1(B(x,r))}\leq C \|
V_{M} - V \|_{L^1_{U}(\R^N)}\to 0$ as $M\to \infty$ and therefore, for
some  $M$ we have  $\sup_{x\in\R^N}\| V_{M} - V \|_{L^1(B(x,r))} <
\frac{c}{2}$. This implies that for all $x\in \R^{N}$
\begin{displaymath}
\int_{B(x,r)} V_{M} =
\int_{B(x,r)} V
+
\int_{B(x,r)} (V_{M} -V)
> c - \frac{c}{2} = \frac{c}{2}.
\end{displaymath}
Again  \cite[Proposition 1.4]{Arendt-Batty} gives (i).
\end{proof}

\begin{remark}
  Now we give an example that shows that in general in Proposition
  \ref{prop:Arendt-Batty_4_bounded_approximations} (ii) does not imply
  (i). That is, we show that may have  for all $M>0$ and $r>0$
  \begin{displaymath}
\inf_{x\in \R^N} \int_{B(x,r)} V_{M} =0
\end{displaymath}
but for some $r>0$,
\begin{displaymath}
\inf_{x\in \R^N} \int_{B(x,r)} V >0 .
\end{displaymath}

Denote by $\{Q_{i}\}$ the family of cubes centered at points $i\in
\Z^{N}$ of
integer coordinates in $\R^{N}$ and with edges of length 1 parallel to
the axes. Then notice that $B(i, 1) \subset Q_{i}$  and define in $Q_{i}$ for $i\neq 0$,
\begin{displaymath}
  V(x) =  c  |i|^{N} \Car_{B(i, \frac{1}{|i|})}
\end{displaymath}
and $V= \frac{1}{c}  \Car_{B(0, 1)}$ in $Q_{0}$ where $c$ is   the measure of the unit
ball.

Then
\begin{enumerate}
\item
  For every $i$, $\int_{Q_{i}} V = 1$ so
  $V\in L^{1}_{U}(\R^{N})$.

  \item
 There exists an $r>0$ such that for every $x\in \R^{N}$, $B(x,r)$
  contains at least one cube $Q_{i}$. Then
  \begin{displaymath}
    \int_{B(x,r)} V \geq 1
  \end{displaymath}

  \item
 For every $M>0$ and every $r>0$ there exists $m(r)\in \N$ such
  that $B(x,r)$ is contained in $m(r)$ cubes and if $|x|$ is large
  then the centers of these cubes satisfy $|i|\geq |x|-2r>0$. For any
  such cubes
  \begin{displaymath}
    \int_{Q_{i}} V_{M} \leq \frac{c M}{|i|^{N}}
  \end{displaymath}
  and then if we take $|x|\to \infty$ we see that
  \begin{displaymath}
    \int_{B(x,r)} V_{M} \leq  m(r) \frac{c M}{(|x|-2r)^{N}} \to 0 .
  \end{displaymath}
\end{enumerate}

\end{remark}

Our next results characterises the exponential type of the fractional
Schrödinger semigroup.

\begin{theorem}\label{thm:decay_and_bottom_spectrumL2}

  With the notations in Theorem
  \ref{thr:exponential_decay_4_fractional}, assume    $0< \mu \leq 1$
  and $0\leq V \in L^{p_{0}}_U(\R^N)$ with
  $p_{0}>\max\{\frac N{2\mu}, 1\}$.

  Then 
\begin{displaymath}
\omega_{2}= a_{*}\mydef \inf\{\int_{\R^{N}} |(-\Delta)^{\frac{\mu}{2}} \phi|^2
    + \int_{\R^{N}} V |\phi|^2 \colon \phi\in C^\infty_c(\R^N), \
    \|\phi\|_{L^2(\R^N)}=1 \}  \geq 0.
  \end{displaymath}

Moreover, if $a_{*}>0$ then
\begin{equation}\label{eq:Arendt-Batty_condition}
\int_{G} V (x) dx=\infty \quad \text{ for all } \ G\in \mathcal{G}.
\end{equation}
Conversely, if $V$  can be approximated in
$L^{1}_{U}(\R^{N})$ by bounded functions, then
(\ref{eq:Arendt-Batty_condition}) implies $a_{*}>0$.

\end{theorem}
\begin{proof}
From Proposition \ref{prop:alternative_contractions_or_decay},  for any $a <\omega_{2}$
  we have that the   semigroup $\{e^{at}
  S_{\mu,V}(t)\}_{t\geq0}$ in $L^2(\R^N)$ is of contractions  and the
  operator $\mathscr{A}=
  - (-\Delta)^\mu -V+ a$ in $L^2(\R^N)$  satisfies 
\begin{displaymath}
\langle -\mathscr{A} \phi, \phi \rangle_{L^2(\R^N)} \geq 0 \quad
\text{ for all } \phi \in C^{\infty}_{c}(\R^N) \subset
D(\mathscr{A}). 
\end{displaymath}

In particular, taking  $\phi\in C^\infty_c(\R^N)$ with  $\|\phi\|_{L^2(\R^N)}=1$ then $\int_{\R^{N}}
|(-\Delta)^{\frac{\mu}{2}} \phi|^2     + \int_{\R^{N}} V |\phi|^2
\geq a >0 $. Hence $a_{*}\geq a$ and therefore $a_{*}\geq
\omega_{2}$.

Conversely, if $a<a_{*}$ we take  $f(t)= e^{2at}
\|u(t)\|_{L^2(\R^N)}^2$ with $u(t)=S_{\mu,V}(t)
\phi$  and then
\begin{displaymath}
  f'(t) = 2a e^{2at} \int_{\R^{N}} u(t)^{2} + 2 e^{2at} \int_{\R^{N}}
  u(t) \frac{d}{dt}u(t) =  2a e^{2at} \int_{\R^{N}} u(t)^{2} + 2 e^{2at} \int_{\R^{N}}
  u(t) (-(-\Delta)^{\mu} u(t) - V u(t))
\end{displaymath}
and then
\begin{displaymath}
  \frac{e^{-2a t}}{2} f'(t)  = -  \int_{\R^{N}} |(-\Delta)^{\frac{\mu}{2}} u(t)|^{2} +
  Vu(t)^{2} + a \int_{\R^{N}} u(t)^{2} \leq (a-a_{*}) \int_{\R^{N}}
  u(t)^{2} \leq 0
\end{displaymath}
which yields $ f'(t)\leq 0$ and then $\|S_{\mu, V} (t)\|_{\mathcal{L}(L^2(\R^N))} \leq
e^{-at}$.
Therefore $\omega_{2}\geq a_{*}$.

Now, if $a_{*}>0$, assume that (\ref{eq:Arendt-Batty_condition})  fails, that is, there exists $G\in \mathcal{G}$
such that
\begin{displaymath}
\int_G V <\infty .
\end{displaymath}
Then we choose a positive sequence satisfying
$r_n> n^{\frac{1}{\mu}}$ and a sequence of points $\{x_n\}\subset \R^N$ such that
$B(x_n,r_n)\subset G$ for each $n\in\N$.

Assume we had  a sequence of functions such that $\{\phi_n\}\subset
C^\infty_c(B(x_n,r_n))$, $\|\phi_n\|_{L^2(\R^N)}=1$ such that
\begin{displaymath}
\lim_{n\to\infty} \|\phi_n\|_{L^\infty(\R^N)}=0 , \quad \lim_{n\to\infty}\|(-\Delta)^\frac{\mu}{2}
\phi_n\|_{L^2(\R^N)}=0.
\end{displaymath}
Then we would  have
\begin{displaymath}
\begin{split}
0< a_*
    &\leq  \int_{\R^N} |(-\Delta)^{\frac{\mu}{2}} \phi_n|^2
    + \int_{\R^N} V |\phi_n|^2
    = \int_{\R^N} |(-\Delta)^{\frac{\mu}{2}} \phi_n|^2
    + \int_{G} V |\phi_n|^2
    \\&
    \leq
    \int_{\R^N} |(-\Delta)^{\frac{\mu}{2}} \phi_n|^2
    + \|\phi_n\|_{L^\infty(\R^N)}^2 \int_{G} V \ \to 0 \ \text{ as } \ n\to \infty,
\end{split}
\end{displaymath}
which is a contradiction and therefore (\ref{eq:Arendt-Batty_condition}) must be true.
To construct a  sequence as above we  take $\psi\in C^\infty_c(B(0,1))$ such that
$\int_{\R^N} |\psi|^2 =1$ and defining 
\begin{displaymath}
\psi_n(x)= n^{-\frac{N}{2\mu}} \psi (n^{-\frac{1}{\mu}} x), \quad x\in \R^N, \ n\in \N
\end{displaymath}
we see that $\{\psi_n\}\subset C^\infty_c(B(0,r_n))$,
$\|\psi_n\|_{L^2(\R^N)}=1$, $\lim_{n\to\infty}
\|\psi_n\|_{L^\infty(\R^N)}=0$ and 
\begin{displaymath}
\lim_{n\to\infty}\|(-\Delta)^\frac{\mu}{2}
\psi_n\|_{L^2(\R^N)}=\lim_{n\to\infty} n^{-1} \|(-\Delta)^\frac{\mu}{2}
\psi\|_{L^2(\R^N)} = 0 
\end{displaymath}
since $(-\Delta)^\frac{\mu}{2}$ is an homogenous operator of degree
$\mu$.
Then $\phi_n(\cdot) = \psi_n(\cdot+x_n)$,  $n\in \N$, satisfies the
conditions above since $(-\Delta)^\frac{\mu}{2}$ is invariant
under translations, see (\ref{eq:fractional_laplacian_nonlocal}). 

Conversely, if $V$  can be approximated in
$L^{1}_{U}(\R^{N})$ by bounded functions, then from Proposition
\ref{prop:Arendt-Batty_4_bounded_approximations} and Theorem
\ref{thr:exponential_decay_4_fractional} we get that
(\ref{eq:Arendt-Batty_condition}) implies $a_{*}>0$.
\end{proof}

\bibliographystyle{plain}

\end{document}